\documentclass[a4paper,11pt,reqno]{amsart}

\usepackage{amssymb}
\usepackage{times}
\usepackage{mhequ}
\usepackage{mhsymb}
\usepackage{amssymb}
\usepackage{mathrsfs}
\usepackage{tikz}
\usepackage{microtype}

\usetikzlibrary{shapes.misc}
\usetikzlibrary{shapes.symbols}
\tikzset{
	dot/.style={circle,fill=black,draw=black, solid,inner sep=0pt,minimum size=0.5mm},
	yy/.style={circle,fill=gray!20,draw=black,inner sep=0pt,minimum size=0.8mm},
	>=stealth,
	}
\makeatletter
\def\DeclareSymbol#1#2#3{\expandafter\gdef\csname MH@symb@#1\endcsname{\tikz[baseline=#2,scale=0.15]{#3}}%
\expandafter\gdef\csname MH@symb@#1s\endcsname{\scalebox{0.6}{\tikz[baseline=#2,scale=0.15]{#3}}}}
\def\<#1>{\csname MH@symb@#1\endcsname}
\makeatother

\DeclareSymbol{X}{-2.4}{\node[dot] {};}
\DeclareSymbol{1}{0}{\draw[white] (-.4,0) -- (.4,0); \draw (0,0)  -- (0,1.2) node[dot] {};}
\DeclareSymbol{2}{0}{\draw (-0.5,1.2) node[dot] {} -- (0,0) -- (0.5,1.2) node[dot] {};}
\DeclareSymbol{3}{0}{\draw (0,0) -- (0,1.2) node[dot] {}; \draw (-.7,1) node[dot] {} -- (0,0) -- (.7,1) node[dot] {};}
\DeclareSymbol{31}{-3}{\draw (0,0) -- (0,-1) -- (1,0) node[dot] {}; \draw (0,0) -- (0,1.2) node[dot] {}; \draw (-.7,1) node[dot] {} -- (0,0) -- (.7,1) node[dot] {};}
\DeclareSymbol{30}{-3}{\draw (0,0) -- (0,-1); \draw (0,0) -- (0,1.2) node[dot] {}; \draw (-.7,1) node[dot] {} -- (0,0) -- (.7,1) node[dot] {};}
\DeclareSymbol{32}{-3}{\draw (0,0) -- (0,-1) -- (1,0) node[dot] {}; \draw (0,0) -- (0,-1) -- (-1,0) node[dot] {}; \draw (0,0) -- (0,1.2) node[dot] {}; \draw (-.7,1) node[dot] {} -- (0,0) -- (.7,1) node[dot] {};}
\DeclareSymbol{22}{-3}{\draw (0,0.3) -- (0,-1) -- (1,0) node[dot] {}; \draw (0,0.3) -- (0,-1) -- (-1,0) node[dot] {};\draw (-.7,1) node[dot] {} -- (0,0.3) -- (.7,1) node[dot] {};}
\DeclareSymbol{20}{-3}{\draw (0,0) -- (0,-1);\draw (-.7,1) node[dot] {} -- (0,0) -- (.7,1) node[dot] {};}
\DeclareSymbol{32W}{-3}{\draw  [densely dotted, semithick] (-1.5,0)node[dot] {}  -- (0,-1.5) node[yy]{} -- (1.5,0) node[dot] {}; \draw (.8,-2.3)node{\tiny $y$} ;\draw[semithick] (0,0) -- (0,-1.5) node[yy] {}; \draw [densely dotted, semithick] (0,0) -- (0,1.8) node[dot] {}; \draw[densely dotted, semithick] (-1,1.5) node[dot] {} -- (0,0) -- (1,1.5) node[dot] {};}
\DeclareSymbol{32WW}{-3}{\draw  [densely dotted, semithick] (-1.5,0)node[dot] {}  -- (0,-1.5) node[yy]{} -- (1.5,0) node[dot] {}; \draw (.8,-2.3)node{\tiny $\bar{y}$} ;\draw[semithick] (0,0) -- (0,-1.5) node[yy] {}; \draw [densely dotted, semithick] (0,0) -- (0,1.8) node[dot] {}; \draw[densely dotted, semithick] (-1,1.5) node[dot] {} -- (0,0) -- (1,1.5) node[dot] {};}

\let\emptyset \undefined
\newsymbol\emptyset    203F

\newcommand{\un}{u_{\mathrm{det}}}

\theoremstyle{plain}
\newtheorem{theorem}{Theorem}[section]

\newtheorem{lemma}[theorem]{Lemma}

\newtheorem{definition}[theorem]{Definition}

\theoremstyle{remark}
\newtheorem{remark}[theorem]{Remark}

\numberwithin{equation}{section}

\newcommand{\T}{\mathbf{T}}

\newcommand{\Cc}{\mathcal{C}}

\newcommand{\Hh}{\mathcal{H}}

\newcommand{\Oo}{\mathcal{O}}

\newcommand{\FF}{\mathbf{\Psi}}
\newcommand{\FFF}{\Psi}

\newcommand{\EE}{\mathbf{E}}

\newcommand{\sz}{\mathbf{s}}

\def\MM{\mathscr{M}}

\def\hom{\mathrm{hom}}

\let\I\CI
\let\J\CJ

\def\MM{\mathscr{M}}
\def\MMm{\mathscr{M}_{\min}}
\def\II{\mathscr{I}}
\def\TT{\mathscr{T}}

\def\GGamma#1{\Gamma_{\!#1}}
\def\bGGamma#1{\bar \Gamma_{\!#1}}
\def\bbar#1{\bar{\bar #1}}
\let\1\one
\def\PPi{\mathbf{\Pi}}
\def\${|\!|\!|}
\def\min{{\tikz[baseline=-1,line cap=round] \draw(0,0)--(0.15,0);}}

\DeclareMathOperator{\dist}{dist}

\begin{document}

\title[Large deviations]
{Large deviations for white-noise driven, nonlinear stochastic PDEs in two and three dimensions}

\author{Martin Hairer}
\address{Martin Hairer, University of Warwick
}
\email{M.Hairer@Warwick.ac.uk}

\author{Hendrik Weber}
\address{Hendrik Weber, University of Warwick
}
\email{hendrik.weber@warwick.ac.uk}





\begin{abstract}
We study the stochastic Allen-Cahn equation driven by a noise term with intensity $\sqrt{\eps}$ and correlation length $\delta$ in two and three spatial dimensions. We study diagonal limits $\delta, \eps \to 0$ and describe fully the large deviation behaviour depending on the relationship between $\delta$ and $\eps$. 

The recently developed theory of regularity structures allows to fully analyse the behaviour of solutions for vanishing correlation length $\delta$ and fixed noise intensity $\eps$.  One key fact is that in order to get non-trivial limits as $\delta \to 0$, it is necessary to introduce diverging counterterms. The theory of regularity structures allows to rigorously analyse this renormalisation procedure for a number of interesting equations.

Our main result is a large deviation principle for these renormalised solutions. One interesting feature of this result is that the diverging renormalisation constants disappear at the level of the large deviations rate function. We apply this result to derive a sharp condition on $\delta, \eps$ that guarantees a large deviation principle for  diagonal schemes $\eps, \delta \to 0$ for the equation without renormalisation.\\

\noindent\textsc{R\'esum\'e.} Nous \'etudions l'\'equation d'Allen-Cahn stochastique conduite par un bruit d'intensit\'e $\sqrt{\eps}$ et de longueur de corr\'elation $\delta$ en dimensions spatiales deux et trois. Nous consid\'erons la limite $\delta,\eps \to 0$ et nous d\'ecrivons compl\`etement le comportement des grandes d\'eviations associ\'ees, suivant les relations entre $\delta$ et $\eps$.

La th\'eorie des structures de r\'egularit\'e r\'ecemment d\'evelopp\'ee permet d'analyser le comportement des solutions \`a intensit\'e de bruit $\eps$ fix\'ee dans la limite $\delta \to 0$. Un fait crucial est que, afin d'obtenir des limites non-triviales dans cette limite, il est n\'ecessaire d'introduire des contretermes divergents. La th\'eorie des structures de r\'egularit\'e permet d'analyser rigoureusement de telles proc\'edures de renormalisation pour un nombre d'\'equations int\'eressantes.

Notre r\'esultat principal est un principe de grandes d\'eviations pour ces \'equations renormalis\'ees. Il est alors int\'eressant de noter que les constantes de renormalisation divergentes disparaissent au niveau de la fonction de taux. Une cons\'equence de ce r\'esultat est une condition optimale sur le comportement relatif de $\delta$ et $\eps$ qui garantit l'existence d'un principe de grandes d\'eviations \'egalement pour l'\'equation non-renormalis\'ee dans certains r\'egimes.
\end{abstract}


\date\today

\maketitle

\section{Introduction}
\label{s:Intro}

The purpose of this article is to provide large deviation results for a class of non-linear stochastic 
PDEs which are driven by space-time white noise, in   $2$ and $3$ spatial dimensions. 
We are going to study equations of the type
\begin{equ}[e:AC]
\d_t u = \Delta u + C u - u^3 + \sqrt{\eps} \xi_\delta\;,
\end{equ}
where $\xi_\delta$ is some random driving noise (we will be more specific very soon) and  $C\in \R$.
In order to avoid complications coming from the effect of boundary conditions, we will always
assume that the spatial variable takes values in the $d$-dimensional torus $\T^d$ and  
that the initial datum $u_0$ is fixed.

Equations of this type are popular as phenomenological descriptions in  various situations, e.g. for phase separation (see e.g. \cite{HH}).
From a physical point of view, the interesting regime is that where the noise is weak (i.e.\ the typical strength of the
noise, say when tested against smooth test functions, is of order $\sqrt \eps \ll 1$)
and almost white (i.e.\ correlations of $\xi$ decay on a lengthscale $\delta \ll 1$). 
In this case, we denote the noise by $\sqrt \eps \xi_\delta$.

In one spatial dimension the solution theory for equation \eqref{e:AC} is well-understood, even in the case of vanishing correlation length $\delta =0$, i.e.\ in the case where the noise term $\xi$ is a space-time white noise.   
In this situation the limit of vanishing noise strength was described in detail by \cite{JL} on the level of large deviation estimates.

Recently, there have been several works 
dedicated to studying \eqref{e:AC} driven by $\sqrt \eps \xi_\delta$ for $\eps \ll \delta \ll 1$ in arbitrary spatial dimension. For fixed $\delta > 0$, the law of \eref{e:AC} 
satisfies a large deviation principle as $\eps \to 0$
with rate $\eps$ and some rate function $\II_\delta$. Formally, as $\delta \to 0$,
the sequence of rate functions $\II_\delta$ converges to a limiting
``rate function'' $\II$ given by
\begin{equ}
\II(u) =  \frac12 \int_0^T \!\!\int_{\T^d} \big( \partial_t u -  \Delta u - C u + u^3 \big)^2 \, dx \, dt   \;.
\end{equ}
(With the understanding that $\II(u)$ is infinite if the distribution $\partial_t u -  \Delta u - C u + u^3$ is not
represented by a square integrable function or if $u$ does not satisfy the initial condition.)
Minimisers of the rate function $\II$ subject to certain initial and terminal conditions were investigated in \cite{ERV,KohnOtto}.  In  \cite{SandraLD} Cerrai and Freidlin showed that the convergence $\II_\delta \to \II$  holds in the sense of 
$\Gamma$-convergence (with respect to a suitable topology) for arbitrary dimension $d$. 
They used this result to conclude $\Gamma$-convergence of the associated  \emph{quasi potentials}. 

These results naturally suggest that $\II$ should be the rate function for a problem
which does not involve $\delta$ anymore. It would be very natural to interpret $\II$ as the rate function for the solutions to
\eref{e:AC} with noise $\sqrt \eps \xi$ where $\xi$ denotes space-time white noise. (In one spatial dimensions $d=1$ it is shown in \cite{JL} that this is indeed true). 
The problem that immediately presents itself is that while the result of \cite{SandraLD} 
holds in any dimension, the equation \eref{e:AC} (with fixed $C$) 
driven by space-time white noise 
is ill-posed in any dimension $d \ge 2$. As a matter of fact, if we denote by $u^{(\eps)}_{\delta}$ the  solution to \eref{e:AC} driven by $\sqrt \eps \xi_\delta$ and by $\un$ the deterministic solution to the equation with $\xi = 0$, then one has
\begin{equ}
\lim_{\delta \to 0} \lim_{\eps \to 0} u^{(\eps)}_{\delta}= \un\;.
\end{equ}
On the other hand, it was shown in \cite{Marc} that, already in dimension $2$, one has
$\lim_{\delta \to 0}u^{(\eps)}_{\delta}= 0$ in a space of distributions for any fixed value of $\eps$! 

There is, however, a way to obtain a non-trivial stochastic limit in a scheme with $\delta \to 0$, but the equation \eqref{e:AC} has to be modified: Consider the solution  $\hat{u}_{\delta}^{(\eps)}$ to
\begin{equ}[e:ACapprox]
\d_t u= \Delta u + \bigl(C + 3\eps C_\delta^{(1)} - 9\eps^2 C_\delta^{(2)}\bigr) u- u^3 + \sqrt \eps\xi_\delta\;,
\end{equ}
where the spatial variable takes values in the $d$-dimensional torus with $d \in \{2,3\}$. 
In two spatial dimensions it was shown in \cite{dPD} that for every fixed $\eps$ the solutions to \eqref{e:ACapprox} have a non-trivial limit as $\delta \to 0$, provided that the $\delta$-depend constants are suitably chosen. Recently, in \cite{Regularity}, the first named author developed a theory of regularity structures allowing to construct a non-trivial limit also in three spatial dimensions. We denote these limits by $\hat{u}^{(\eps)}$.

\begin{remark}
In dimension $d=2$, one can actually take $C_\delta^{(2)} = 0$ and 
$C_\delta^{(1)} = {1\over 4\pi} |\log \delta|$.
In dimension $d=3$ on the other hand, one has $C_\delta^{(1)} \propto \delta^{-1}$ with a proportionality
constant depending on the details of the regularisation of $\xi_\delta$, see \eref{e:CWick}.
Furthermore, it is $C_\delta^{(2)}$ which should be taken proportional to $|\log \delta|$ in this case. 
\end{remark}

The purpose of this article is to investigate the large deviations of the solutions obtained from the scheme \eqref{e:ACapprox} as  $\eps \to 0$.
Our main result, Theorem~\ref{thm:main1}, states that for every $T>0$ and every function $\eps \mapsto \delta(\eps) \geq 0$ with 
$\lim_{\eps \to 0} \delta(\eps) = 0$, the laws of the sequence $\hat{u}_{\delta(\eps)}^{(\eps)}$ satisfy
a large deviation principle in the space $\CC([0,T], \CC^{1-d/2-\kappa})$ 
(for arbitrarily small $\kappa >0$) with rate $\eps$ 
and rate function $\II$ (see below in Section~\ref{s:RegStr}  for a discussion of spaces with negative regularity). This result includes the case $\delta =0$: The processes $\hat{u}^{(\eps)}$ do not depend on a correlation length $\delta$ and satisfy a large deviation principle with rate function $\II$. A remarkable feature of this result is that the \emph{diverging renormalisation constants} $C_\delta^{(1)}$ and $C_\delta^{(2)}$, present in \eqref{e:ACapprox}, disappear on the level of the large deviations, independently of the relationship between $\delta$ and $\eps$. 

As an immediate consequence of our method, in Theorem~\ref{thm:main2}, we make more precise the condition on $\delta$ and $\eps$ under which the solutions $u^{(\eps)}_{\delta}$ of the equation \eqref{e:AC} 
without renormalisation constants satisfy a large deviation principle. Assume that the function $\delta(\eps)$ satisfies
\begin{equ}
\lim_{\eps \to 0}  \eps \delta(\eps)^{-1} = \lambda^2 \in [0,\infty],
\end{equ}
for $d=3$, or $\lim_{\eps \to 0}  \eps \log\big(\delta(\eps)^{-1}\big) = \lambda^2$ for $d=2$. If $\lambda=0$  the solutions $u^{(\eps)}_{\delta}$  also satisfy a large deviation principle with rate function $\II$. If $\lambda \in (0,\infty)$, 
we still obtain a large deviation principle, but this time with a \emph{modified} rate function that depends on $\lambda$. (And on
the regularisation of $\xi_\delta$ in dimension $3$.)

The threshold $\eps \ll  \big( \log(\delta^{-1}) \big)^{-1}$ for $d =2$ and $\eps \ll \delta $ for $d =3$  is not surprising. Indeed, 
it is quite straightforward  to check without reference to the theory of regularity structures that if this condition is satisfied, 
then  as $\eps \to 0$ the solutions $u^{(\eps)}_{\delta(\eps)}$ converge in probability to $\un$, with respect to the topology of uniform 
convergence (see \cite[Section 4]{Marc} for the argument if $d=2$). Indeed, the same statement holds for arbitrary spatial dimension 
$d \geq 4$ provided $\eps \ll \delta^{d-2}$. It is however important to note that our method to prove the large deviation principle 
relies strongly on the understanding of the renormalised equation \emph{even for the schemes without renormalisation}. 
In particular, it is not clear whether a large deviations principle holds in higher dimensions, even in the regime $\eps \ll \delta^{d-2}$.

We would also like to emphasise that the methods developed in this article are not restricted to the 
specific equation \eqref{e:AC}. 
The same arguments would yield similar results for all equations that can be treated with the methods 
developed in \cite{Regularity}. 
In particular, in $d=2$ the nonlinear term $u^3$ in \eqref{e:AC} could be replaced by an arbitrary 
polynomial of odd degree with negative leading-order coefficient, and one has
a large deviations principle for the KPZ equation driven by small noise. 
In the same way we could obtain a large deviation principle for the 
two dimensional Navier-Stokes equations driven by space-time white noise 
(see \cite{daPratoDebusscheNavier} for the solution theory at $\eps>0$ fixed, see \cite{bouchet2014langevin} for an analysis of the corresponding rate function). 

Finally, we point out that partial results in two spatial dimension were already obtained by Jona-Lasinio and Mitter in \cite{JonaLasino}.  In \cite{Aida2,Aida1} Aida studied a related model in one spatial dimension, but with scaling properties akin to the two-dimensional
case. We also want to mention that the theory of regularity structures grew out of an attempt to generalise the theory of \emph{rough paths} to higher dimensions. Large derivation results for rough paths were obtained by \cite{LedouxRP, MilletSanz,FrizVictoir}.

\subsection{Structure of the article}
In Section~\ref{s:RegStr} we present a very short summary of some notions of the theory of regularity structures. 
In Section~\ref{s:WienerLDP} we discuss random variables in a fixed Wiener chaos taking values in a separable Banach 
space and their large deviations. Section~\ref{s:AC-LDP} contains the statements and proofs of our main results.

\subsection*{Acknowledgements}

{\small
We are grateful to the anonymous referee for carefully reading the original manuscript and 
making various suggestions that improved the exposition.
MH was supported by the Leverhulme trust through a leadership award and
the Royal Society through a Wolfson research award.
MH would also like to thank the Institute for Advanced Study 
for its warm hospitality and the `The Fund for Math' for funding his stay there, over the course
of which this work was completed.
HW was supported by an EPSRC First Grant. 
}

\section{Regularity structures}
\label{s:RegStr}

In order to prove the type of convergence result mentioned in the 
introduction, we make use of the theory of regularity structures 
developed in \cite{Regularity}. A complete self-contained exposition of
the theory is of course beyond the scope of this article, so we will
content ourselves with a short summary of the theory's main concepts 
and results, when specialised to the specific example of the stochastic
Allen-Cahn equation \eref{e:AC}.
For a concise exposition of the general theory, see also the lecture 
notes \cite{RegularityLN}.

The main ingredient of the theory is that of a \textit{regularity structure}.
In our case, this consists of a graded vector space $\CT = \bigoplus_{\alpha \in A} \CT_\alpha$
where $A$ denotes a set of real-valued indices (called homogeneities) that is 
locally finite and bounded from below. Each of the spaces $\CT_\alpha$ will be finite-dimensional
and come with a distinguished canonical basis. The space $\CT$ also comes
endowed with a group $\CG$ of continuous linear transformations of $\CT$ with the property that,
for every $\Gamma \in \CG$, every $\alpha \in A$, and every $\tau \in \CT_\alpha$ one has
\begin{equ}[e:basicRel]
\Gamma \tau - \tau \in \CT_{<\alpha} \eqdef \bigoplus_{\beta < \alpha} \CT_\beta\;.
\end{equ}

The canonical example to keep in mind is the space $\bar \CT = \bigoplus_{n \in \N} \bar \CT_n$ 
of abstract polynomials in
finitely many indeterminates, with $A = \N$ and $\bar \CT_n$ denoting the space of monomials
that are homogeneous of degree $n$. In this case, a natural group of transformations $\CG$ acting
on $\bar \CT$ is given by the group of translations, which does indeed satisfy \eref{e:basicRel}.

\subsection{Specific regularity structure}
\label{sec:regStruc}

The regularity structure that is relevant for the analysis of \eref{e:AC} is built in the
following way. First, we start with the regularity structure $\bar \CT$ given by all polynomials in
$d+1$ indeterminates, let us call them $X_0,\ldots,X_d$, which denote the time
and space directions respectively. We do however endow these with the parabolic space-time 
scaling instead of the more usual Euclidean scaling 
so that each factor of the ``time'' variable $X_0$ increases the homogeneity by $2$.
More precisely, one has $\one \in \bar \CT_0$,  $X_0 \in \bar \CT_2$, $X_i \in \bar \CT_1$
for $i \in \{1,\ldots,d\}$, etc.

We then introduce two additional symbols, $\Xi$ and $\CI$, which will be interpreted 
as an abstract representation of the driving noise $\xi$ and of the operation of convolution
with the heat kernel respectively. Fixing some (sufficiently small in the sequel) exponent $\kappa > 0$,
we then \textit{postulate} that $\Xi$ has homogeneity $|\Xi| = -{d+2\over 2} - \kappa$
and, if $\tau$ is some formal expression with homogeneity $|\tau| = \alpha$, then
$\CI(\tau)$ is a new formal expression with homogeneity $|\CI(\tau)| = \alpha + 2$.
We also \textit{postulate} that $\CI(X^k) = 0$ for every multi-index $k$, which will make sense in
view of \eref{e:killPoly} below.
(Here, for $k = (k_0,\ldots,k_d)$, we have used the shorthand $X^k = X_0^{k_0}\cdots X_d^{k_d}$.)
Furthermore, if $\tau, \bar \tau$ are formal expressions with respective homogeneities
$\alpha, \bar \alpha$, then $\tau \bar \tau = \bar \tau \tau$ is postulated to be a new formal expression with
homogeneity $\alpha + \bar \alpha$. 

A few examples of formal expression with their respective homogeneities
that can in principle be built in this way are given by
\begin{equ}[e:examples]
|\CI(\Xi)^2| = 2-d-2\kappa\;,\quad 
|\CI(\CI(\Xi)^3)| = 5 -{3d\over 2} - 3\kappa\;,\quad 
|\Xi\CI(\Xi)| = -d-2\kappa\;. 
\end{equ}
In order to define our regularity structure $\CT$, we do not keep all of these formal expressions,
but only those that are actually useful for the abstract reformulation of \eref{e:AC}.
More precisely, we consider a collection $\CU$ of formal expressions which is the
smallest collection containing $\one$, $X$, and $\CI(\Xi)$, and such that  
\begin{equ}
\tau_1,\tau_2,\tau_3 \in \CU \quad\Rightarrow\quad \CI(\tau_1\tau_2\tau_3) \in \CU\;.
\end{equ}
Here  and below we use $X$ to denote the collection of all $X_i$ for $i\in \{ 0,1, \ldots,d \}$.
We then set 
\begin{equ}
\CW = \{\Xi\} \cup \{\tau_1\tau_2\tau_3\,:\, \tau_i \in \CU\}\;,
\end{equ}
and we define $\CT$ as the set of all linear combinations of elements in $\CW$. 
(Note that since $\one \in \CU$, one does in particular have $\CU \subset \CW$.)
Naturally, $\CT_\alpha$ consists of those linear combinations that only involve
elements in $\CW$ that are of homogeneity $\alpha$. Furthermore, $\CW$ is the
previously announced set of canonical basis elements of $\CT$.
In particular, $\CT$ contains the first
two formal expressions of \eref{e:examples}, but not the last one. It follows furthermore
from \cite[Lem.~8.10]{Regularity} that, for every $\alpha \in \R$, $\CW$ contains only
finitely many elements of homogeneity less than $\alpha$, so that each $\CT_\alpha$
is finite-dimensional.

In order to simplify expressions later, we will use the following shorthand graphical 
notation for elements of $\CW$. For $\Xi$, we simply draw a dot.
The integration map is then represented by a downfacing line and the multiplication of 
symbols is obtained by joining them at the root. For example, we have
\begin{equ}[e:trees1]
\CI(\Xi) = \<1>\;,\quad
\CI(\Xi)^3 = \<3>\;,\quad
\CI(\Xi)\CI(\CI(\Xi)^3) = \<31>\;.
\end{equ}
Symbols containing factors of $X$ have no particular graphical representation.

\subsection{Structure group}

Let us now describe the structure group $\CG$ associated to the space $\CT$. For this, we 
first introduce additional symbols $\CJ_k(\tau)$ and denote by
 $\CT_+$, the free commutative algebra generated by 
\begin{equ}[e:genT+]
\CW_+ \eqdef \bigl\{X\bigr\}\cup \bigl\{\CJ_k(\tau)\,:\, \tau \in \CW\;, \; |\tau| + 2 > |k|\}\;,
\end{equ}
where $k$ is an arbitrary $(d+1)$-dimensional multi-index and $|k|$ denotes its
``parabolic length'', i.e.
\begin{equ}
|k| = 2k_0 + k_1 +\ldots + k_d\;.
\end{equ}
In other words, $\CT_+$ consists of all linear combinations of products of formal expressions 
in $\CW_+$. We will view $\CJ_k$ as a map from $\CT$ into $\CT_+$
by postulating that it acts linearly on $\CT$ and that $\CJ_k(\tau) = 0$ for those formal expressions
$\tau$ for which $|\tau| + 2 \le |k|$.

With this definition at hand, we construct a linear map $\Delta \colon \CT \to \CT\otimes \CT_+$ in a recursive way. 
In order to streamline notations, we shall write 
$\tau^{(1)}\otimes\tau^{(2)}$ as a shorthand for $\Delta\tau$.
(This is an abuse of notation, following Sweedler, since in general
$\Delta\tau$ is a linear combination of such terms.) We then define $\Delta$ via the identities
\begin{equ}
\Delta\1=\1\otimes\1\;,\qquad
\Delta\Xi= \Xi\otimes\1\;,\qquad
\Delta X_i= X_i\otimes\1+\1\otimes X_i\;,
\end{equ}
and then recursively by the following relations:
\begin{equs}
\Delta \tau\overline{\tau}&= \tau^{(1)}\overline{\tau}^{(1)}\otimes \tau^{(2)}\overline{\tau}^{(2)}\;,\\
\Delta\I(\tau)&=\I(\tau^{(1)})\otimes \tau^{(2)}+\sum_{\ell,m}\frac{X^\ell}{\ell!}\otimes\frac{X^m}{m!}
\J_{\ell+m}(\tau)\;.
\end{equs}
For any linear functional $f \colon \CT_+ \to \R$, we can now define in a natural way
a map $\Gamma_{\!f} \colon \CT \to \CT$ by
\begin{equ}[e:Gammaf]
\Gamma_{\!f} \tau = (I \otimes f)\Delta \tau\;,
\end{equ}
where $I$ denotes the identity map.
Let now $\CG_+$ denote the set of all such linear functionals  which are multiplicative in the sense that 
$f(\tau \bar \tau) = f(\tau)f(\bar \tau)$ for any two elements $\tau, \bar \tau \in \CT_+$. With this definition
at hand, we set
\begin{equ}
\CG = \{\Gamma_{\! f}\,:\, f \in \CG_+\}\;.
\end{equ}
It is not difficult to see that these operators satisfy the property \eref{e:basicRel}, but it is
a non-trivial fact that the set $\CG$ of these linear operators does indeed form a group
under composition, see \cite[Sec.~8.1]{Regularity}.

\begin{remark}\label{rem:cutoff}
As a matter of fact, we will never need to consider the full space $\CT$ as defined above,
but it will be sufficient to consider the subspace generated by all elements of homogeneity less
than some  large enough number $\zeta$. In practice, it actually turns out to be 
sufficient to choose $\zeta = 2$.
\end{remark}

\begin{remark}
The construction we just explained depends on the dimension $d$, so we should really use the notations
$\CT^{(d)}$, $\CT_+^{(d)}$, etc. In order to keep the notations simple, we refrain from doing so. 
Our statements will always either hold for both $d=2,3$, or if they do not, then 
the value of $d$ for which they hold will be made clear. 
\end{remark}

\subsection{Models}\label{ss:models}

Now that we have fixed our algebraic regularity structure $(\CT,\CG)$, we introduce a family of analytical 
objects associated to it that will play the role of Taylor polynomials in our theory in order to
allow us to describe solutions to \eref{e:AC} locally, up to arbitrarily high order, despite the fact
that they are not smooth in the conventional sense.

From now on, we also fix a value $\zeta \ge 2$ as in Remark~\ref{rem:cutoff} and we
set $\CT = \CT_{<\zeta}$. This has the advantage that
$\CT$ is finite-dimensional so we do not need to worry about topologies.
We first fix a kernel $K \colon \R^{d+1} \to \R$ with the following properties:
\begin{enumerate}
\item The kernel $K$ is compactly supported in $\{|x|^2 + |t| \le 1\}$.
\item One has $K(t,x) = 0$ for $t \le 0$ and $K(t,-x) = K(t,x)$.
\item For $(t,x)$ with $|x|^2 + t < 1/2$ and $t > 0$, one has
\begin{equ}
K(t,x) = {1\over |4\pi t|^{d/2}} e^{-{|x|^2 \over 4t}}\;,
\end{equ}
and $K$ is smooth on $\{|x|^2 + |t| \ge 1/4\}$.
\item For every polynomial $P \colon \R^{d+1} \to \R$ of parabolic degree less than $\zeta$, one has
\begin{equ}[e:killPoly]
\int_{\R^{d+1}} K(t,x) P(t,x)\,dx\,dt = 0\;;
\end{equ}
\end{enumerate}
in other words, $K$ has essentially all the properties of the heat kernel, except that it is furthermore
compactly supported and satisfies \eref{e:killPoly}. The constants $1/2$ and $1/4$ appearing in
the third point are of course completely arbitrary as long as they are strictly between $0$ and $1$.
The existence of a kernel $K$ satisfying these properties is very easy to show.

We now denote by $\CS'$ the space of distributions on $\R^{d+1}$ and by $\CL(\CT,\CS')$ the space
of (necessarily continuous) linear maps from $\CT$ to $\CS'$. Furthermore, given
a continuous test function $\phi\colon \R^{d+1} \to \R$ and a point $z = (t,x) \in \R^{d+1}$, we set
\begin{equ}[e:rescale]
\phi_z^\lambda(\bar z) = \lambda^{-(d+2)} \phi\bigl(\lambda^{-2}(\bar t - t), \lambda^{-1}(\bar x - x)\bigr)\;,
\end{equ}
where we also used the shorthand $\bar z = (\bar t, \bar x)$. Finally, we write $\CB$ for the set
of functions $\phi \colon \R^{d+1} \to \R$ that are smooth, compactly supported in the ball of radius one,
and with their values and derivatives up to order $3$ bounded by $1$.

Given a kernel $K$ as above, we then introduce a set $\MM$ of \textit{admissible models} which
will play the role of Taylor polynomials for our solution theory.
An admissible model consists of a pair $(\Pi,F)$ of functions
\begin{equs}[2]
\Pi \colon \R^{d+1} &\to \CL(\CT,\CS') \quad & \quad F \colon \R^{d+1}  &\to \CG \\
z &\mapsto \Pi_z & z &\mapsto F_z 
\end{equs}
with the following properties. 
First, writing $\gamma_{z\bar z} \in \CG_+$ for the element
such that $F_z^{-1} \circ F_{\bar z}  = \GGamma{\gamma_{z\bar z}}$, they satisfy the analytical bounds
\begin{equ}[e:bounds]
\bigl|\bigl(\Pi_z \tau\bigr)(\phi_z^\lambda)\bigr| \lesssim \lambda^{|\tau|}\;,\qquad
\bigl|\gamma_{z\bar z}(\bar \tau)\bigr| \lesssim |z-\bar z|^{|\bar \tau|}\;,
\end{equ}
uniformly over $\phi \in \CB$, $\lambda \in (0,1]$, $\tau \in \CW$, and $\bar \tau \in \CW_+$. 
Here, with the same shorthand as before,
we set
\begin{equ}
|z-\bar z| = |x-\bar x| + \sqrt{|t-\bar t|}\;.
\end{equ}
The proportionality constants implicit in the notation $\lesssim$ of \eref{e:bounds} are assumed to be bounded
uniformly for $z$ and $\bar z$ taking values in any compact set. We also assume
that one has the algebraic identity
\begin{equ}[e:algebraic]
\Pi_{z} F_z^{-1} = \Pi_{\bar z} F_{\bar z}^{-1}\;,
\end{equ}
valid for every $z, \bar z$ in $\R^{d+1}$. Finally, and this is why our models are called
\textit{admissible}, we assume that $\bigl(\Pi_z \one\bigr)(\bar z) = 1$, that for every multi-index $k$
and every $\tau \in \CW$ with $X^k \tau \in \CT$,
\minilab{e:admissible}
\begin{equ}\label{e:admissible1}
\bigl(\Pi_z X^k \tau\bigr)(\cdot) = (\cdot- z)^k\Pi_z \tau \;,\qquad f_z(X^k) = (-z)^k\;,
\end{equ}
and that, for every $\tau \in \CW$ with $\CI\tau \in \CT$
(recall that we did truncate $\CT$ so this is not true for \textit{all} $\tau$), one has the identities
\minilab{e:admissible}
\begin{equs} 
f_z(\CJ_k \tau) &=  -\int_{\R^{d+1}} D^k K(\bar z - z)\bigl(\Pi_{z} \tau\bigr)(d\bar z) \;, \qquad |k| < |\tau|+2\;,\label{e:admissible2}\\
\bigl(\Pi_z \CI \tau\bigr)(\bar z) &=  \int_{\R^{d+1}} K(\bar z - \bbar z)\bigl(\Pi_{z} \tau\bigr)(d\bbar z) + \sum_{k} {(\bar z - z)^k \over k!} f_z(\CJ_k \tau) \;.\qquad \label{e:admissible3}
\end{equs}
Here, we wrote similarly to above $f_z \in \CG_+$ for the element such that $F_z = \GGamma{f_z}$.
Recall that we have set $\CJ_k \tau = 0$ if $|k| \ge |\tau|+2$, so that the sum appearing
on the second line is automatically always finite.
It is not clear in principle that these integrals converge, but it turns out that the analytic conditions
\eref{e:bounds} guarantee that this is always the case, see \cite[Sec.~5]{Regularity}.

\begin{remark}\label{rem:redundant1}
Since $f_z \in \CG_+$, so that it is multiplicative, \eref{e:admissible1} and \eref{e:admissible2}
do specify $f_z$ (and therefore $F_z$ via \eref{e:Gammaf}) 
completely, once we know $\Pi_z$. There is therefore
quite a lot of rigidity in these definitions, which makes the mere existence of admissible models a 
highly non-trivial fact.
\end{remark}

\begin{remark}\label{rem:redundant2}
Building further on Remark~\ref{rem:redundant1}, it actually turns out that if the map $\Pi \colon \R^{d+1} \to \CL(\CT,\CS')$ 
satisfies the {\em first} analytical bound in \eref{e:bounds} and
is such that, for $F$ defined from $\Pi$ via \eref{e:admissible}, one has the 
identities \eref{e:admissible} and \eref{e:algebraic}, then the second analytical bound in \eref{e:bounds}
is {\em automatically} satisfied. This is a consequence of \cite[Thm.~5.14]{Regularity}.
\end{remark}

Since we will only ever consider \eref{e:AC} with periodic boundary conditions, we will always assume that 
the model $(\Pi,F)$ is periodic in the following sense. We are given $d$ vectors $e_1,\ldots,e_d \in \R^d$ and we
denote by $T_i\colon \R^{d+1} \to \R^{d+1}$ the translation maps
\begin{equ}
T_i \colon (t,x) \mapsto (t,x+e_i)\;.
\end{equ}
We also define the natural action of $T_i$ on test functions by $(T_i \phi)(z) = \phi(T_i^{-1} z)$.
We then say that the model $(\Pi,F)$ is periodic if
\begin{equ}[e:periodic]
\bigl(\Pi_{T_i z} \tau \bigr)(T_i\phi) = \bigl(\Pi_{z} \tau \bigr)(\phi)\;,\qquad F_{T_i z} = F_z\;,
\end{equ}
for every $z \in \R^{d+1}$, every smooth test function $\phi$, every $\tau$, and every $i$.
This definition is consistent in the sense that if $\Pi$ satisfies the first identity in \eref{e:periodic}
and $F$ is given by \eref{e:admissible2}, then $F$ automatically satisfies the second identity.

Given two admissible models $(\Pi,F)$ and $(\bar \Pi, \bar F)$, the bounds \eref{e:bounds} yield a natural
notion of a semi-distance between them by considering, for any given compact domain $D \subset \R^{d+1}$,
the quantity
\begin{equ}[e:distPi]
\$\Pi-\bar \Pi\$ = \sup_{z \in D} \sup_{\phi\in \CB \atop \lambda\in(0,1]} \sup_{\tau \in \CW} \lambda^{-|\tau|}{\bigl|\bigl(\Pi_z \tau - \bar \Pi_z \tau\bigr)(\phi_z^\lambda)\bigr| }\;.
\end{equ}
While this seminorm does of course depend on the domain $D$, we will always take
$D = I \times \R^d$ for some large enough interval $I$. The precise length of $I$ 
does not really matter as long as it contains $[-2,T+2]$, where $T$ is the 
final time for which we wish to obtain our large deviations principle.
We will use the following definition.

\begin{definition}
The space $\MM^{(d)}$ consists of all periodic admissible models for the regularity structure $(\CT^{(d)}, \CG)$
built above in dimension $d\in \{2,3\}$. It is endowed with the metric \eref{e:distPi} for the domain
$D = I \times \R^d$ with a sufficiently large interval $I$.
\end{definition}

\begin{remark}
As defined here, the space $\MM^{(d)}$ is not separable which might lead to technical difficulties.
In practice, these can easily be resolved by defining $\MM^{(d)}$ as the closure of the set of
\textit{smooth} admissible models, i.e. the set of models for which 
$\Pi\colon \R^{d+1} \to \CL(\CT_\min^{(d)},\CC^{\infty}(\R^{d+1}))$ is a smooth function,  under the distance \eref{e:distPi}, which is a separable space. 
\end{remark}

It is extremely important at this stage to note that $\MM^{(d)}$ is in general \textit{much} smaller
than the space of all (periodic in the sense of \eref{e:periodic}) 
functions $\Pi\colon \R^{d+1} \to \CL(\CT,\CS')$ such that \eref{e:distPi}
is finite. This is because we still have the requirement that the identities \eref{e:algebraic}
and \eref{e:admissible} hold. It may also
appear that the ``norm'' \eref{e:distPi} neglects to control the
second bound in \eref{e:bounds} but, by Remark~\ref{rem:redundant2}, this bound actually holds automatically on
$D' = I'\times \R^d$ if we know that $\$\Pi\$$ is finite on $D = I \times \R^d$ for $I$ 
containing a large enough neighbourhood of $I'$. (Essentially $D'$ needs to be large enough
to contain the support of $K(\cdot - z)$ for every $z \in D$ so that one can apply
\eqref{e:admissible2}.)

Given any smooth space-time function $\xi_\delta$ which is periodic in the sense that $T_i\xi_\delta = \xi_\delta$, 
there is a canonical way of lifting it to an admissible periodic model
$\Psi(\xi_\delta) = (\Pi^\delta, F^\delta)$ as follows. First, we set $\Pi^\delta_z \Xi = \xi_\delta$, independently
of $z$, and we define it on $X^k$ as in \eref{e:admissible1}. 
Then, we define $\Pi^\delta_z$ recursively by \eref{e:admissible3}, as well as the additional identity
\begin{equ}[e:canonical]
\bigl(\Pi_z^\delta \tau \bar \tau\bigr)(\bar z) = \bigl(\Pi_z^\delta \tau\bigr)(\bar z)
\bigl(\Pi_z^\delta \bar \tau\bigr)(\bar z)\;.
\end{equ}
Note that this is guaranteed to make sense if $\xi_\delta$ is a function but not for more general distributions! 
It was shown in
\cite[Prop.~8.27]{Regularity} that if we furthermore define $F^\delta$ via \eref{e:admissible2}, then
this does indeed define an admissible model for every continuous function $\xi_\delta$.
It is however very important to keep in mind that not every admissible model is obtained in this way, 
or even as a limit of such models. We will use this fact in Section~\ref{sec:renorm} below.

Since an admissible model $(\Pi,F)$ is interpreted as 
an extension of the usual Taylor polynomials, it is quite natural to define 
spaces $\CD^{\gamma,\eta}$ which mimic a weighted
version of the H\"older spaces $\CC^\gamma$ in the following way.

\begin{definition}
A function $U \colon \R^{d+1} \to \CT_{<\gamma}$ belongs to $\CD^{\gamma,\eta}$ if, for every
compact domain $D$, one has
\begin{equ}[e:normU]
\|U\|_{\gamma,\eta} \eqdef \sup_{z \in D}\sup_{\alpha < \gamma} {\|U(z)\|_\alpha \over |t|^{(\eta - \alpha)\wedge 0}}
+ \sup_{z \neq \bar z\in D \atop |z-\bar z| \le 1}\sup_{\alpha < \gamma} {\|U(z) - \GGamma{z\bar z} U(\bar z)\|_\alpha \over \bigl(|t|\wedge |\bar t|\bigr)^{\eta-\gamma} |z-\bar z|^{\gamma - \alpha}} < \infty\;.
\end{equ}
Here, we wrote $\|\tau\|_\alpha$ for the norm of the component of $\tau$ in $\CT_\alpha$ and we used as before the notation $\GGamma{z\bar z} = F_z^{-1} F_{\bar z}$.
We also used $t$ and $\bar t$ as shorthands for the time components of the space-time points $z$ and $\bar z$.
\end{definition}

\begin{remark}
The powers of $t$ appearing in this definition allow elements of $\CD^{\gamma,\eta}$ to exhibit a
singularity on the line $\{(t,x)\,:\, t = 0\}$. This is essential in order to be able to deal with solutions to
\eref{e:AC} with distributional initial conditions.
\end{remark}

\begin{remark}
In order to streamline notations, we suppressed the dependence on the domain $D$ in this norm.
This is because, similarly to our definition of admissible models,
we will only ever use this on some fixed space-time domain.
\end{remark}

The space $\CD^{\gamma,\eta}$ depends in a crucial way on the underlying model
$(\Pi,F)$. Therefore, it is not obvious \textit{a priori} how to compare elements belonging to $\CD^{\gamma,\eta}$,
but based on two different models. This is however crucial when investigating the continuity of solutions
to \eref{e:AC} as a function of the underlying model.
A natural distance between elements $U \in \CD^{\gamma,\eta}$ and $\bar U \in \bar \CD^{\gamma,\eta}$
(denoting by $\bar \CD^{\gamma,\eta}$ the space built over the model $(\bar \Pi, \bar F)$), is given by \eref{e:normU},
with $U(z)$ replaced by $U(z) - \bar U(z)$ in the first term and $U(z) - \GGamma{z\bar z} U(\bar z)$ replaced by
\begin{equ}
U(z) - \bar U(z) - \GGamma{z\bar z} U(\bar z) +  \bGGamma{z\bar z} \bar U(\bar z)
\end{equ}
in the second term. Note that this distance is \textit{not} a function of $U - \bar U$!

The idea now is to reformulate \eref{e:AC} as a fixed point problem in $\CD^{\gamma,\eta}$ 
(based on the canonical model $\Psi(\xi_\delta)$ built above) for suitable values of the exponents 
$\gamma$ and $\eta$. As a matter of fact, we will view it as a fixed point problem in
the subspace $\CD_\CU^{\gamma,\eta} \subset \CD^{\gamma,\eta}$ consisting of those functions
taking values in the linear span of $\CU$. If we consider an admissible model consisting of smooth
functions, then one can define an operator $\CR$ acting on $\CD^{\gamma,\eta}$ and taking values in the
space of continuous functions by
\begin{equ}[e:defR]
\bigl(\CR U\bigr)(z) = \bigl(\Pi_z U(z)\bigr)(z)\;.
\end{equ}
A remarkable fact, and this is the content of \cite[Thm~3.10]{Regularity}, is that as soon
as $\gamma > 0$ the map
\begin{equ}[e:mapR]
(\Pi,U)\mapsto \CR U \in \CS'
\end{equ}
given by \eref{e:defR} is jointly (locally)
Lipschitz continuous with respect to the metric defined in \eref{e:normU} and \eref{e:distPi}, so that 
it makes sense even in situations where the definition \eref{e:defR} is nonsensical!
This of course relies very heavily on the fact that we \textit{only} consider admissible models in \eref{e:mapR}
and not arbitrary functions $\Pi\colon \R^{d+1} \to \CL(\CT,\CS')$. The map $\CR$ is called the \textit{reconstruction operator}
since it reconstructs the (global) distribution $\CR U$ from the (local) data $U$ and $\Pi$.

\subsection{Continuity properties}\label{ss:continuity}

As we will see below, it turns out that the solution to \eref{e:AC} can be obtained, via a suitable
fixed point argument, as a \textit{jointly continuous} map of the initial condition and a suitable
random model $(\Pi,F)$ built from the underlying space-time white noise. As a consequence of the contraction
principle, it will therefore be sufficient to obtain a suitable large deviations principle for 
the random model. For this, we want to reduce ourselves to the ``minimal'' amount of information
necessary to reconstruct the whole model. Denote by $\CW_\min^{(d)} \subset \CW$ the list of symbols of
negative order in dimension $d$, which do not contain any factor $X_i$. 
In dimension $d=2$ these are given by
\begin{equ}
\CW_\min^{(2)} = \{\Xi, \<1>\ , \<2>\ , \<3>\}\;,
\end{equ}
while in dimension  $d=3$, one has
\begin{equ}[e:CW3m]
\CW_\min^{(3)} = \{\Xi, \<1>\ , \<2>\ , \<3>\ , \<32>\ , \<22>\ , \<31>\}\;.
\end{equ}
(Here and in the sequel, we assume that $\kappa$ is small enough. Otherwise, these lists could expand.)
We also denote by $\CT_\min^{(d)} \subset \CT$ the linear span of $\CW_\min^{(d)}$.
We now define the space $\MMm^{(d)}$ (with $d = 2,3$) of ``minimal admissible models'' 
in the following way. 

It is natural to take for $\MMm^{(d)}$ the closure of the set of smooth maps
$\Pi\colon \R^{d+1} \to \CL(\CT_\min^{(d)},\CC^{\infty}(\R^{d+1}))$ that are periodic in the sense given previously,
under the norm given by \eref{e:boundsM02} below. We furthermore impose that the constraint 
\eqref{e:algebraic} holds with $F_z$ given by \eqref{e:admissible2}.
Besides providing a bound of the type \eref{e:bounds}, we also want to
ensure that $\Pi_z \<1>$ is not just a space-time distribution of order $|\<1>|$,
but actually a continuous function of time with values in a space of distributions of order $|\<1>|$.
In order to formulate this, we denote by $\CB_0$ a set of test functions just like $\CB$, but depending only on
the spatial variable. Their translated and rescaled versions $\phi_x^\lambda$ are defined analogously 
to \eref{e:rescale}. We then assume that $\Pi$ is such that
\begin{equs}[e:boundsM02]
\sup_{\lambda\in (0,1]} \sup_{\phi \in \CB} \sup_{\tau \in \CW_\min^{(d)}} \sup_{z \in D} \lambda^{-|\tau|} \bigl|\scal{\Pi_z \tau,\phi_z^\lambda}\bigr| &< \infty \;, \\
\sup_{\lambda\in (0,1]} \sup_{\phi \in \CB} \sup_{s \in \R} \sup_{z \in D} \lambda^{-|\Xi|} \bigl|\scal{\one_{t \ge s}\Pi_z \Xi,\phi_z^\lambda}\bigr| &< \infty \;,\\
\sup_{\lambda\in (0,1]} \sup_{\phi \in \CB_0} \sup_{z \in D} \lambda^{-|\<1>|} \bigl|\scal{(\Pi_z \<1>)(t,\cdot),\phi_x^\lambda}\bigr| &< \infty \;,
\end{equs}
where $\scal{\cdot,\cdot}$ denotes the usual $L^2$ inner product and $D = I \times \T^2$ for some fixed but
large enough interval $I$.
Furthermore, we impose that the identities \eref{e:algebraic} and \eref{e:admissible} 
hold. In order for this definition to make sense,
we need to make sure that we have enough data to define the linear maps $F_z$ on $\CT_\min^{(d)}$. 
In dimension $d=2$, it is straightforward to verify that $\Delta \tau = \tau \otimes \one$
for every $\tau \in \CW_\min^{(2)}$, so that the action of $\CG$ on $\CT_\min^{(2)}$ is trivial.
The condition \eref{e:algebraic} then reduces to stating that $\Pi_z$ does actually not depend on $z$ at all.
In particular, the space $\MMm^{(2)}$ is a separable Banach space.

In dimension $d=3$ and for $\kappa$ sufficiently small,
we note that one has the identities
\begin{equs}
\Delta \<22> &= \<22>\otimes \one + \<2>\otimes \CJ(\<2>)\;,\quad
\Delta \<31> = \<31>\otimes \one + \<1>\otimes \CJ(\<3>)\;,\\
\Delta \<32> &= \<32>\otimes \one + \<2>\otimes \CJ(\<3>)\;,
\end{equs}
so that the action of $F_z$ on $\CT_\min^{(3)}$ only requires \eref{e:admissible2} with $\tau \in \{\<2>,\<3>\}$.
This data is available from $\Pi_z \<2>$ and $\Pi_z \<3>$ (which actually do not depend on $z$), so that the 
first condition in \eref{e:boundsM02}  is indeed meaningful. 
Note that, unlike  $\MMm^{(2)}$, $\MMm^{(3)}$ is \textit{not} a linear space.
It is however a separable metric space and can be viewed as a closed subset of the separable Banach space
obtained by relaxing \eref{e:algebraic} to the requirement that 
$\Pi_z \tau = \Pi_{\bar z} \tau$ for $\tau \in \CW_\min^{(2)}$.

The reason why the spaces $\MMm^{(d)}$ are important is the following result.

\begin{theorem}\label{theo:extension}
For every $\PPi \in \MMm^{(d)}$ there exists a unique admissible model $(\Pi, F) \in \MM^{(d)}$
such that $\Pi_z \tau = \PPi_z \tau$ for all $\tau \in \CW_\min^{(d)}$ and $z \in \R^{d+1}$.
Furthermore, the map $\PPi \mapsto (\Pi,F)$ is locally Lipschitz continuous. 
\end{theorem}

\begin{proof}
This is an immediate consequence of Proposition~3.31 and Theorem~5.14 of \cite{Regularity}.
\end{proof}

\subsection{Abstract fixed point problem}

We now reformulate \eref{e:AC} as a fixed point problem in
$\CD_\CU^{\gamma,\eta}$ for suitable values of $\gamma$ and $\eta$. Note first that by Duhamel's formula, \eref{e:AC} is equivalent for smooth $\xi$ to the integral
equation
\begin{equ}[e:ACint]
u = P \star \bigl((\xi + C u - u^3)\one_{t > 0}\bigr) + Pu_0\;.
\end{equ}
Here, $P$ denotes the heat kernel, $\star$ denotes space-time convolution, 
and $Pu_0$ denotes the solution to the heat equation with initial condition $u_0$.
A local solution is a pair $(u,\tau)$ with $\tau > 0$ and such that \eref{e:ACint}
holds on $[0,\tau] \times \R^d$. In order to interpret this equation as an identity
in $\CD^{\gamma,\eta}$, we need to find an analogue for the operation of 
convolution against $P$.

It turns out that, provided that $\gamma < \bar \gamma + 2$, that $\eta < |\<1>| = 1-{d\over 2}-\kappa$
and that $\bar \eta > -2$,
it is possible to construct a linear operator 
$\CP \colon \CD^{\bar \gamma,\bar \eta}\to \CD_\CU^{\gamma,\eta}$ 
with the following properties:
\begin{enumerate}
\item One has the identity $\CR \CP U = P \star \CR U$.
\item One has $\CP U = \CI U + \tilde \CP U$\;, where $\tilde \CP U$ only takes values in $\bar \CT$,
the linear span of $\{X^k\}$.
\item There exists $\theta > 0$ such that
\begin{equ}
\|\CP \one_{t > 0}U\|_{\gamma,\eta} \le T^\theta \|U\|_{\bar \gamma, \bar \eta}\;,
\end{equ}
where the norms are taken over the domain $[0,T] \times \R$.
\end{enumerate}
Furthermore, even though $\one_{t > 0}\Xi \not \in \CD^{\bar \gamma,\bar \eta}$
with $\bar \gamma$ and $\bar \eta$ as above, it turns out that if the second bound
in \eref{e:boundsM02} is satisfied (in particular if the model comes from an 
element of $\MMm^{(d)}$ as in Theorem~\ref{theo:extension}), then there exists an element
$\CP \one_{t > 0}\Xi \in \CD^{\gamma,\eta}_\CU$ satisfying all of the properties (1)--(3).
For the precise construction of $\CP$ and a 
proof of these properties, see Section~5, Proposition~6.16, 
and Theorem~7.1 in \cite{Regularity}.

Finally, given a function $u$ of (parabolic)
class $\CC^\gamma$, we write $\TT_\gamma u$ for its Taylor expansion
of (parabolic) order $\gamma$, namely 
\begin{equ}
\bigl(\TT_\gamma u\bigr)(z) = \sum_{|k| < \gamma} {X^k \over k!} \bigl(D^k u\bigr)(z) \in \bar \CT \subset \CT\;. 
\end{equ}
It was then shown in \cite[Lemma~7.5]{Regularity} that one has $\TT_\gamma P u_0\in \CD^{\gamma,\eta}_\CU$,
provided that $u_0 \in \CC^\alpha(\T^d)$ for some $\alpha > \eta$. Note that this is true 
for arbitrary $\eta \in \R$, in particular one can have $\eta < 0$.

With all of these notations at hand, we can lift \eref{e:ACint} in a very natural way
to a fixed point problem in $\CD_\CU^{\gamma,\eta}$, by looking for solutions $U$ to
\begin{equ}[e:FP]
U = \CP  \one_{t > 0} \bigl(\Xi + C U - U^3\bigr) + \TT_\gamma P u_0\;.
\end{equ}
According to \cite[Prop.~6.12]{Regularity}  for $U \in \CD_\CU^{\gamma,\eta}$ we can actually define $U^3$ as a modelled distribution. 
Similarly to \eref{e:ACint}, we call a pair $(U,\tau)$ with $\tau > 0$ a local
solution if the identity \eref{e:FP} holds on $[0,\tau) \times \R^d$. 

\begin{remark}
For \eref{e:FP} to make sense, we need a suitable definition for the element 
$\CP  \one_{t > 0}\Xi$. As already mentioned, this is the case if the underlying
model comes from an element of $\MMm^{(d)}$.
\end{remark}

We will say that local solutions are unique if, whenever $(U,\tau)$ and
$(\bar U,\bar \tau)$ are local solutions, one has $U(t,x) = \bar U(t,x)$ for
$t \in (0, \tau \wedge \bar \tau]$. Note also that if $U \in \CD_\CU^{\gamma,\eta}$
with $\gamma > 0$ and the underlying model comes from an element of $\MMm^{(d)}$,
then $\CR U$ is a continuous function of time with values in the space 
$\CC^{\beta} = \CC^{\beta}(\T^d)$ of periodic distributions in space of regularity $\beta = |\<1>| = 1-{d\over 2} - \kappa$ (see for example \cite[Def. 3.7]{Regularity} for a definition of $\CC^{\beta}$ for $\beta<0$. On a bounded domain, it agrees with the Besov space $\CB^{\beta}_{\infty, \infty}$).
Given a final time $T>0$, we say that a local solution 
$(U,\tau)$ to \eref{e:FP} is \textit{maximal} if either $\tau = T$ or
$\lim_{t \to \tau} \|(\CR U)(t,\cdot)\|_{\beta} = \infty$.
Given a final time $T$, it will be convenient to view solutions as elements of the
space $\CX = \CC([0,T], \CC^{\eta}) \cup \{\infty\}$ for some $\eta < \beta$, 
which is a metric space by postulating that the ``cemetery state'' $\infty$ is at 
distance $1$ of all other elements.

One then has the following result, where we implicitly assume that
$\kappa$ is sufficiently small ($\kappa < 1/10$ is sufficient for example).

\begin{theorem}\label{theo:gen}
Let $\gamma \in (1+4\kappa,\zeta)$, and let $\eta \in (-{2\over 3}, 1-{d\over 2}-\kappa)$. Then, for every admissible model $(\Pi,F)$ arising from an element of $\MMm^{(d)}$, \eref{e:FP} admits a unique maximal solution
in $\CD_\CU^{\gamma,\eta}$.

Furthermore, if there exists a smooth function $\xi_\delta$ such that $(\Pi,F) = \Psi(\xi_\delta)$, then
this solution is global and the function $u = \CR U$ solves \eref{e:ACint}. Finally, define the 
solution map $\CS_A\colon \R \times \CC^\eta \times \MMm^{(d)} \to \CX$ which maps $(C,u_0,\PPi)$
onto $\CR U$ if the maximal solution $U$ to \eqref{e:FP} satisfies $\tau > T$ and to $\infty$ otherwise.
Then, for every $Z = (C,u_0,\PPi)$ such that $\CS_A(Z) \neq \infty$, there exists a neighbourhood $\CO$ of 
$Z$ such that $\CS_A$ is Lipschitz continuous from $\CO$ to $\CX$.
\end{theorem}

\begin{proof}
This is a synthesis of the results of \cite[Sec.~7]{Regularity} and \cite[Sec.~9]{Regularity},
see also \cite{ICM} for an informal overview of the proof. The continuous dependence on $C$ is
not stated explicitly there, but is implicit in the proofs.

The only fact which remains to be shown is that the map $(\Pi,F) \mapsto \CR U$ is (locally) 
continuous with values in $\CX$. Indeed, the reconstruction operator $\CR$ only maps the space
$\CD_\CU^{\gamma,\eta}$ into $\CC^{\eta}([0,T]\times \T^d)$ in general \cite[Prop.~6.9]{Regularity}, 
which is strictly larger than $\CX$. 
However, one does get continuity of $\CS_A$ into $\CX$. The reason why this is true is that 
by \eqref{e:ACint} the solution $\CR U$ is of the form
\begin{equs}
\CR U = P \star \bigl(\xi + C \CR U - \CR(U^3))\one_{t > 0}\bigr) + Pu_0\;,
\end{equs}
with $\CR$ the reconstruction operator. 
Since $U$ takes values in $\CD_\CU^{\gamma,\eta}$, it follows from \cite[Prop.~6.12]{Regularity} 
that $U^3$ takes values 
in $\CD^{\gamma',3\eta}$ for some $\gamma'> 0$. By \cite[Prop. 6.9]{Regularity}, one 
therefore obtains that $\CR(U^3) \in \CC^{3(1 - \frac{d}{2} -\kappa)}([0,T]\times \T^d)$, which in particular shows 
by usual (parabolic) Schauder estimates that $P \star \CR(U^3)$ takes values in $\CC([0,T]\times \T^d) \subset \CX$. Combining this with the fact that both $Pu_0$ and  $P \star \xi$ belong to $\CX$ (the latter
as a consequence of the second and third bound in \eqref{e:boundsM02}) concludes the proof.
\end{proof}

\subsection{Renormalised model}
\label{sec:renorm}

It turns out that, if we take for $\xi_\delta$ some mollification of space-time white noise then
the sequence $\Psi(\xi_\delta)$, with the canonical lift $\Psi$ defined via \eref{e:canonical},
does \textit{not} converge in $\MM^{(d)}$ (and a fortiori not in $\MMm^{(d)}$). Instead, we fix two renormalisation constants 
$C_\delta^{(1)}$ and $C_\delta^{(2)}$ and we define a ``renormalised lift'' $\hat{\Psi}_\delta$ in the
following way. Regarding $\hat{\Pi}^\delta_z \<1>$ (which actually doesn't depend on $z$), we
define it from $\hat{\Pi}^\delta_z \Xi = \xi_\delta$ via \eref{e:admissible3}.

In dimension $d=2$, we then replace \eref{e:canonical} by the identities
\begin{equ}[e:Wick]
\hat\Pi^\delta_z \<2> = \bigl(\hat\Pi^\delta_z \<1>\bigr)^2 - C_\delta^{(1)}\;,\quad
\hat\Pi^\delta_z \<3> = \bigl(\hat\Pi^\delta_z \<1>\bigr)^3 - 3C_\delta^{(1)}\hat\Pi^\delta_z \<1>\;.
\end{equ}
We then use Theorem~\ref{theo:extension} to extend this to all of $\CT^{(2)}$.

In dimension $d=3$, we also define $\Pi^\delta_z$ on $\<2>$ and $\<3>$ as in \eref{e:Wick},
but we furthermore set
\begin{equs}
\hat\Pi^\delta_z \<22> &= \bigl(\hat\Pi^\delta_z \<20>\bigr)\bigl(\hat\Pi^\delta_z \<2>\bigr) - C_\delta^{(2)}\;,\quad
\hat\Pi^\delta_z \<32> = \bigl(\hat\Pi^\delta_z \<30>\bigr)\bigl(\hat\Pi^\delta_z \<2>\bigr) - 3C_\delta^{(2)} \hat\Pi^\delta_z\<1>\;,\\
\hat\Pi^\delta_z \<31> &= \bigl(\hat\Pi^\delta_z \<30>\bigr)\bigl(\hat\Pi^\delta_z \<1>\bigr)\;,\label{e:3dWick}
\end{equs}
where $\hat\Pi^\delta_z \<20>$ and $\hat\Pi^\delta_z \<30>$ are defined from
$\hat\Pi^\delta_z \<2>$ and $\hat\Pi^\delta_z \<3>$ through \eref{e:admissible3}.
Again, we then use Theorem~\ref{theo:extension} to extend this to all of $\CT^{(3)}$.
Here, the non-trivial fact that $\hat\Pi^\delta$ defined in this way indeed belongs to 
$\MMm^{(3)}$ was shown in Theorem~8.44 and Section~9.2 of \cite{Regularity}.

As usual, we define $\hat F^\delta_z$ through \eref{e:admissible2} and we write $(\hat \Pi^\delta, \hat F^\delta) = \hat{\Psi}_\delta(\xi_\delta)$. 
With these definitions, the following theorem is the main result of \cite[Sec.~10]{Regularity}.

\begin{theorem}\label{theo:mainregularity}
Let $\rho \colon \R^{d+1} \to \R$ be a compactly supported smooth function with $\int \rho = 1$ and 
set
\begin{equ}
\rho_\delta(t,x) = \delta^{-d-2} \rho(\delta^{-2}t, \delta^{-1}x)\;,\qquad \xi_\delta = \rho_\delta \star \xi\;,
\end{equ}
where $\xi$ is periodic space-time white noise.

There exist choices of constants $C_\delta^{(1)}$ and $C_\delta^{(2)}$ such that the sequence
$\hat{\Psi}_\delta(\xi_\delta)$ converges in probability in $\MMm^{(d)}$ to a limiting model $\hat \Pi$.
Furthermore, the function $u = \CR U$, where $U$ is the (local) solution to \eref{e:FP} 
with model $\hat \Pi^\delta$ is the classical solution to 
\begin{equ}
\d_t u = \Delta u + (3 C_\delta^{(1)} - 9 C_\delta^{(2)}) u - u^3 + \xi_\delta\;.
\end{equ}
\end{theorem}

\begin{remark}\label{rem:rem214}
In dimension $2$, we do not need $C_\delta^{(2)}$ (i.e.\ we set it equal to $0$), 
and $C_\delta^{(1)}$ can be chosen  as
\begin{equ}[e:2DWicklog]
C_\delta^{(1)} = {1\over 4\pi} \log \delta^{-1} + c_\rho \;,
\end{equ}
for some constant $c_\rho$ of order $1$ which depends on the mollifier $\rho$. 
The precise expression for $c_\rho$ does not matter at this stage.

In dimension $3$ on the other hand, one should take 
\begin{equ}[e:CWick]
C_\delta^{(1)} = \delta^{-1} \int_{\R^4} \bigl(P\star \rho\bigr)(t,x)^2\,dt\,dx\;,
\end{equ}
where $\star$ denotes space-time convolution as before and $P$ denotes the heat kernel (without
any truncation). 
This time, it is the constant $C_\delta^{(2)}$ that has to be chosen to diverge logarithmically:
\begin{equ}[e:CWick7]
C_\delta^{(2)} = c\log \delta^{-1} + \bar c_\rho\;,
\end{equ}
with $c$ a universal constant independent of $\rho$ and $\bar c_\rho$ some constant that depends on the
mollifier. The precise values for $c$ and $\bar c_\rho$ are irrelevant for our current purpose. 
\end{remark}

\section{Large deviations for Wiener chaos}
\label{s:WienerLDP}
In this section we state and prove a large deviation statement for random variables taking values in a fixed 
inhomogeneous Wiener chaos over an arbitrary Gaussian probability space. These results are then applied in Section~\ref{s:AC-LDP} to obtain a large deviation result for the 
renormalised models constructed from a white noise with small amplitude.

Large deviation results for random variables in a fixed Wiener chaos have been studied by several authors, see e.g. \cite{Borell1, Borell2, Ledoux,Nualart}. For the reader's convenience, 
we give a complete exposition in our context. The core of our argument is a \emph{generalised contraction principle} in the spirit of \cite[Lemma 2.1.4]{DS}.
We start by recalling the following definition from \cite{DS}.

\begin{definition}\label{def:LDP}
Let $(S,d)$ be a separable metric space. We say that a family $\{ \mu_\eps, \eps >0 \}$ of probability measures on $S$ satisfies a large deviation principle with rate $\eps$ and rate function $\II$ if
\begin{enumerate}
\item $\II \colon S \to [0,+\infty]$ is lower semicontinuous, has compact sublevel sets and is not identical to $+\infty$. 
\item For every closed set $\Cc \subseteq S$ we have
\begin{equ}[e:LDPupper0]
 \limsup_{\eps \to 0} \eps \log \mu_\eps \big(\Cc \big) \leq  -\inf_{s \in \Cc} \II(s)\;.
\end{equ}
\item For every open subset $\Oo \subseteq S$  we have 
\begin{equ}[e:LDlower0]
 \liminf_{\eps \to 0} \eps \log \mu_\eps \big( \Oo \big) \geq  -\inf_{s \in \Oo} \II(s)\;.
\end{equ}
\end{enumerate}
We say that a family of $S$-valued random variables  satisfies a large deviation principle if their distributions do. 
\end{definition}
\begin{remark}
All of the large deviation principles in this paper have rate $\eps$ and we will not repeat the rate every time. 
\end{remark}
We also recall the following version of the contraction principle:
\begin{lemma}\label{le:contraction1}
Let $\{\mu_\eps,\eps>0\}$ be a family of probability measures on a separable metric space $(S,d)$ and let $\II\colon S \to [0, \infty]$ satisfy the first condition of Definition~\ref{def:LDP} . Let $(S',d')$ be another separable metric space, and let $\{{\FFF}_\eps, \eps \geq 0\}$ be a family of mappings from $S$ to $S'$ which are continuous on a neighbourhood of $\{s \in S \colon \II(s) < \infty \}$. We assume that 
\begin{enumerate}
\item The probability measures $\{\mu_\eps, \eps >0 \}$ satisfy a large deviation principle on $S$ with rate function $\II$. 
\item For every $c \in \R$, there exists a neighbourhood $\CO_c$ of $\{s \in S\,:\,\II(s) \le c\}$ such that the mappings ${\FFF}_\eps$ converge uniformly on $\CO_c$ to $\FFF_0$.
\end{enumerate}
Then the image measures $\mu_\eps^{-1} \circ {\FFF}_{\eps}^{-1}$ satisfy a large deviation principle on $S'$ with rate function 
\begin{equ}
\II'(s') = \inf\{\II(s) \colon \; s \in S, \; \FFF_0(s) = s' \}\;,    
\end{equ}
with the convention that the $\inf$ equals $+\infty$ if the set is empty.
\end{lemma}
\begin{proof}
This is a direct consequence of \cite[Exercise 2.1.20 (ii)]{DS} (see also \cite[Lemma 2.1.4]{DS}). 
Indeed, in this exercise it is shown that the conclusion of the theorem holds if the $\FFF_\eps$ are continuous on all of $S$ and if  for all $\alpha >0$ we have
\begin{equ}
\limsup_{\eps \to 0 }\eps \log \mu_\eps \Big(\xi \colon d' \big( \FFF_\eps (\xi), \FFF_0(\xi) \big) > \alpha  \Big) = - \infty\;.
\end{equ}

Furthermore, given any $c \in \R$, the uniform convergence of ${\FFF}_\eps$ on $\CO_c$ implies that 
for every $\alpha > 0$ one can find $\eps_0$ such that for $\eps < \eps_0$ the condition
$d'(\FFF_\eps(s), \FFF_0(s)) > \alpha$ implies $s \not\in \CO_c$. The LDP for $\mu_\eps$ then implies
that $\limsup_{\eps \to 0} \eps \log \mu_\eps(\{s\,:\, d'(\FFF_\eps(s), \FFF_0(s)) > \alpha\}) \le -c$
for arbitrary $c$, so that this limit has to be $-\infty$.

Furthermore, the condition that the $\FFF_\eps$ should be continuous everywhere can be replaced by continuity in an open neighbourhood of $\{s \in S \colon \, \II(s) < \infty \}$  -- indeed, just apply the above reasoning to the measures $\mu_\eps$ conditioned on this open neighbourhood of $\{s \in S \colon \, \II(s) < \infty \}$ on which the functions $\FFF_\eps$ are continuous. These conditioned measures have the same large deviation behaviour as the original measures. 
\end{proof}

Now we proceed to the discussion of large deviations for Banach-valued Wiener chaos. Let $(B, \Hh, \mu)$ be an abstract Wiener space, i.e.\ $\mu$ is a centred Gaussian measure on a separable Banach space $B$ with Cameron-Martin space $\Hh \subset B$. Let $(e_i)_{i=1}^\infty$ be an orthonormal basis of $\Hh$ consisting of elements of the dual space $B^{\star}$, where $B^\star \subset \CH$ is defined by dualising
the inclusion $\CH \subset B$ and identifying $\CH^\star$ with $\CH$. (Such a basis can always be found.) Then for every $e_i$ the linear mapping $\xi \mapsto \langle \xi, e_i \rangle $ is continuous on $B$  and defines a centred Gaussian random variable with $\int \langle \xi,e_i \rangle \, \langle \xi , e_j \rangle \, \mu (d\xi) = \delta_{ij}$.

Recall that for $\xi \in \R$ the Hermite polynomials $H_k(\xi)$ are defined by
\begin{equ}
\exp\Big( \lambda \xi -\frac12 \lambda^2  \Big) = \sum_{k=0}^\infty {\lambda^k \over \sqrt{k!}} H_k(\xi)\;,
\end{equ}
i.e.\ $H_0(\xi)  = 1$, $H_1(\xi) =\xi$, $H_2(\xi) = \frac{1}{\sqrt 2} (\xi^2 -1)$, $H_3(\xi) = \frac{1}{\sqrt 6} (\xi^3 - 3\xi)$ etc.

\begin{remark}
The convention used here is that if $\xi$ is a centred normal random variable with variance $1$,
then the random variables $H_k(\xi)$ are orthonormal in $L^2(\Omega)$.
\end{remark}

We call multi-index any sequence $\alpha = (\alpha_1, \alpha_2, \ldots ) \in \N_0^\N$ with only finitely many non-zero entries. We use the convention  $\alpha! = \prod_{i=1}^\infty \big( \alpha_i !\big) $ and $|\alpha| = \sum_{i=1}^\infty \alpha_i$.

Then for any multi-index $\alpha$  we define the $\R$-valued function
\begin{equ}
H_\alpha (\xi) =  \prod_{i=1}^\infty H_{\alpha_i} \big(  \langle \xi, e_i \rangle \big)\;  \qquad \text{for }\xi \in B.
\end{equ} 
For $k \in \N$ the $k$-th homogeneous $\R$-valued Wiener chaos $\Hh^{(k)}(\mu) = \Hh^{(k)}(\mu,\R)  $ is defined as the closure in $L^2(\mu,\R)$ of the linear space generated by the functions $H_\alpha(\xi)$ with $|\alpha| =k$.

Let $E$ be a real separable Banach space and denote by $L^p(\mu,E)$ ($p \in [1, \infty)$) the space of all $E$-valued measurable functions $\FFF$ on $(B,\mu)$ with $\int \| \FFF(\xi) \|_E^p \, \mu (d \xi) < \infty$. For any $k \in \N$ the $k$-th $E$-valued Wiener chaos $\Hh^{(k)}(\mu,E)$  is defined as the closure in $L^2(\mu,E)$ of the linear space generated by random variables of the form $H_\alpha(\xi) \, y$ 
with $|\alpha| = k$ and $y \in E$. Equivalently, one can set 
\begin{equ}[e:character2]
\Hh^{(k)}(\mu,E)  \eqdef \Big\{ \FFF \in L^2(\mu, E) \colon \int \FFF(\xi) \, H_\alpha(\xi) \, \mu (d \xi) =0 \, \text{for all $|\alpha| \neq k$ }  \Big\}.
\end{equ}
Indeed, it is obvious that each random variable $\FFF(\xi) = \sum_{|\alpha| = k} y_\alpha H_\alpha(\xi)$ with only finitely many 
summands satisfies the condition specified in \eqref{e:character2} and that this condition is stable under convergence in $L^2(\mu,E)$. On the other hand, if $\FFF \in L^2(\mu,E)$ satisfies the condition of \eqref{e:character2}, then the conditional expectation of $\FFF$ with respect to the $\sigma$-algebra $\CF_N \eqdef \sigma(\langle \xi, e_1\rangle, \ldots, \langle \xi, e_N\rangle )$ is given by 
\begin{equ}[e:series]
\Psi_N(\xi) \eqdef \sum_{\substack{|\alpha| =k\\ 
 \alpha_i = 0 ,\, i >N}}  \bigg( \int \FFF(\xi')  \, H_\alpha(\xi')   \mu (d \xi') \bigg) \,  H_\alpha(\xi)  \;,
\end{equ}
and  the  convergence theorem for Banach-valued martingales shows that this series converges to $\FFF$ in $L^2(\mu, E)$ as $N$ goes to infinity. Actually, the strong integrability properties of random variables in a fixed Wiener chaos (see e.g. \cite[Thm. 4.1]{Borell3}) imply the much stronger statement that there exists a constant $C$ and a sequence $\beta_N $ with $\lim \beta_N \to \infty$ such that 
\begin{equ}[e:WienerIntegrability]
\int \exp \Big( \beta_N \| \FFF(\xi) - {\FFF}_N(\xi) \|_E^{\frac{2}{k}}   \Big)  \mu( d \xi)  \leq C\;.
\end{equ}

It is important to note here that even though we are using what looks like an $L^2$ theory, $E$ is \textit{not} assumed to
be a Hilbert space. As a matter of fact, this point is absolutely crucial to our argument.

Fix now a
 \begin{equ}[e:Fsum]
 \FFF(\xi) = \sum_{|\alpha|=k} y_\alpha\, H_\alpha(\xi)
 \end{equ}
in $ \Hh^{(k)}(\mu,E) $. Denote by ${\FFF}_N$ the same sum restricted to those indices $\alpha$ with $\alpha_i =0$ for $i >N$. For any $h \in \Hh$ we can define the \emph{homogeneous part} of $\FFF$ by
\begin{equ}[e:fhom]
{\FFF}_{\hom}(h)= \int \FFF(\xi+h) \, \mu(d\xi) \,.
\end{equ}
The Cameron-Martin theorem implies that for every $h \in \CH$ the random variable $\FFF(\xi +h)$ is well defined $\mu$-almost surely and that the expectation in \eqref{e:fhom} converges. For the truncated sum ${\FFF}_N$ it is easy to check that 
\begin{equ}
\big({\FFF}_N \big)_{\hom}(h) \eqdef  \int {\FFF}_N(\xi+h) \, \mu(d\xi)  = \sum_{\substack{|\alpha| =k\\ 
\alpha_i = 0 , \, i >N}} y_\alpha h^\alpha,
\end{equ}
where $h^\alpha \eqdef \prod_{i=1}^\infty \langle h,e_i \rangle^{\alpha_i}$. On the other hand we have by the Cameron-Martin theorem and Cauchy-Schwarz inequality
\begin{equs}
\big\| \big({\FFF}_N \big)_{\hom}(h)&  - {\FFF}_{\hom}(h)  \big\|_E\\
& =   \Big\| \int \big(  {\FFF}_N (\xi +h) - \FFF(\xi+h)  \big) \, \mu(d\xi) \Big\|_E \label{e:Fhomconvergence}\\
&= \Big\| \int  \exp \Big( \langle h, \xi \rangle - \frac12 | h |_{\CH}^2 \Big)   \big(  {\FFF}_N (\xi ) - \FFF(\xi)\big)\;  \mu(d \xi) \Big\|_E  \\
&\leq \exp\Big(\frac{1}{2}| h |_{\CH}^2 \Big) \Big\|  {\FFF}_N (\xi ) - \FFF(\xi) \Big\|_{L^2(\mu, E)} \;, 
\end{equs}
which implies uniform convergence of the sum 
\begin{equ}
{\FFF}_{\hom}(h) = \sum_{|\alpha| =k} y_\alpha h^\alpha
\end{equ}
on bounded subsets of $\CH$.

From now on, we will consider random variables $\FF = \bigoplus_{\tau \in \CW} {\FFF}_{\tau}$, where $\CW$ is some finite index set and  each component ${\FFF}_\tau$ lies in an $E_\tau$-valued inhomogeneous Wiener chaos of order $K_\tau$. This means that for each $\tau$ we can write
\begin{equ}
{\FFF}_\tau = \sum_{k=0}^{K_\tau} {\FFF}_{\tau,k} \;,
\end{equ}
where ${\FFF}_{\tau,k} \in \CH^{(k)}(\mu, E_\tau)$ and $E_\tau$ is some separable Banach space. Let $(\FF_{\delta}, \delta \in (0,1))$ be a family of  such random variables. We are going to say that $\FF_{\delta}$ converges to $\FF$ if for each $\tau \in \CW$ and $k \leq K_\tau$
\begin{equ}[e:convergenceWiener]
\lim_{\delta \to 0 }\int \big\| {\FFF}_{\delta;\tau,k}(\xi) - {\FFF}_{\tau,k}(\xi) \big\|_{E_\tau}^2  \,\mu (d\xi) =0 \;.
\end{equ}
We also define the homogenous part of $\FF$ as $\FF_{\hom} = \bigoplus_{\tau \in \CW} ({\FFF}_{\tau, K_\tau} )_{\hom}$. Note that 
in this definition there is no contribution corresponding to the  ${\FFF}_{\tau,k}$ for $k < K_\tau$ to $\FF_{\hom}$.
The reason for this is that this respects the scaling used in \eqref{e:scaling}.

The main result of this section is a large deviation statement for the random variables $\FF$ when changing the \emph{noise intensity}. The natural rescaling  of the different coordinates in our context is given by    
\begin{equ}[e:scaling]
\FF^{(\eps)}(\xi) \eqdef \bigoplus_{\tau \in \CW} \, \eps^{\frac{K_{\tau}}{2}} {\FFF}_\tau(\xi) \;.
\end{equ}

The main result of this section states that under the measure $\mu$, the  random variables $\FF^{(\eps)} $ satisfy a large deviation principle that is stable under convergence to $\FF$. For any $\sz = \bigoplus_{\tau \in \CW}s_\tau \in \EE \eqdef\bigoplus_{\tau \in \CW} E_\tau$, the rate function is given by 
\begin{equ}[e:rateWiener]
\II(\sz) = \inf \bigg\{ \frac12 |h|^2_{\CH} \colon h \in \Hh \text{ with }  \FF_{\hom}(h) = \sz  \bigg\}\;,
\end{equ}
where the infimum is interpreted as $+\infty$ if the set is empty.
\begin{theorem}[Large deviations in Wiener chaos]\label{thm:wienerLDP}
Assume that $\FF_\delta$ converges to $\FF$ as $\delta$ tends to $0$ in the sense of \eqref{e:convergenceWiener}.  Choose $\delta = \delta(\eps) \ge 0$ such that $\lim_{\eps \downarrow 0} \delta(\eps) = 0$. Then the random variables $\FF_{\delta(\eps)}^{(\eps)}(\xi)$ satisfy a large deviation principle on $\EE$ with rate $\eps$ and rate function $\II$ given by \eqref{e:rateWiener}.
\end{theorem}
\begin{remark}\label{rem:scaling}
The reason that the lower order contributions ${\FFF}_{\tau,k}$ for $k< K_\tau$ vanish on the level of the large deviation principle is the scaling \eqref{e:scaling}. In order to have a contribution from ${\FFF}_{\tau,k}$ one would have to consider $\eps^{\frac{k}{2}} {\FFF}_{\tau,k}$ instead of  $\eps^\frac{{K_\tau}}{2} {\FFF}_{\tau,k}$.
\end{remark}

As stated above, our argument follows roughly the strategy employed in \cite{DS} to prove the classical Freidlin-Wentzell estimates (see \cite[Sec. 1.4 and Lemma 2.1.4]{DS}). 

Suppose that for every $\delta \geq 0$, and every $\tau \in \CW$ and $k \leq K_\tau$ we have 
\begin{equ}[e:series2]
{\FFF}_{\delta;\tau,k}(\xi) = \sum_{|\alpha| = k}  y_{\alpha; \delta; \tau,k} \, H_\alpha(\xi)\;.
\end{equ}
Denote as before by $\FF_{N;\delta} (\xi)$ the conditional expectation of  $\FF_{\delta}(\xi)$ with respect to $\CF_N$, i.e.\ the series expansion \eref{e:series2}  of each ${\FFF}_{\delta; \tau,k }(\xi)$ is restricted to indices $\alpha$ with $\alpha_i=0$ for $i >N$.

\begin{lemma}\label{le:LDPFN}
For any fixed $N \in \N$ the random variables $\FF_{N,\delta(\eps)}^{(\eps)}(\xi)$ satisfy a large deviation principle with rate function 
\begin{equ}[e:defIIn]
\II_N(\sz) = \inf \bigg\{ \frac12 |h|_{\CH}^2 \colon h \in \Hh \text{ with }  (\FF_N)_{\hom}(h) = \sz  \bigg\}\;.
\end{equ}

\end{lemma}
\begin{proof}
As stated above, the random functions ${\FFF}_{N;\delta; \tau,k}$ are continuous functions $B \to E_\tau$  for every fixed $N$. 

Note that for every $\tau \in \CW$ and $k \leq K_\tau$ and for every multi-index $\alpha$ with $|\alpha|=k$ we have by Cauchy-Schwarz inequality
\begin{equs}
\|  y_{\alpha;\delta;\tau,k} - y_{\alpha;\tau,k} \|_{E_\tau}^2 &= \Big\| \int    \big( {\FFF}_{\delta;\tau,k}(\xi) - {\FFF}_{\tau,k}(\xi) \big) H_\alpha(\xi)   \,   \mu(d \xi) \Big\|_{E_\tau}^2 \\
&\leq  \int \big\| {\FFF}_{\delta;\tau,k}(\xi) - {\FFF}_{\tau,k}(\xi) \big\|_{E_\tau}^2  \,\mu (d\xi) \;.\label{e:yconverge}
\end{equs}
This converges to zero as $\delta \to 0$ by assumption. 

In order to apply the contraction principle, Lemma \ref{le:contraction1}, it is useful to rewrite the rescaled random variables as
\begin{equ}
\eps^{\frac{K_{\tau}}{2}} {\FFF}_{N;\delta(\eps);\tau,k}(\xi)  = \Phi^{(\eps)}_{N; \tau,k}(\eps^{\frac12} \xi)\;,
\end{equ}
where 
\begin{equ}
\Phi^{(\eps)}_{N;\tau,k} (\xi) \eqdef \eps^{\frac{K_\tau -k}{2}}\sum_{\substack{|\alpha| =k\\ \alpha_i =0,\, i > N}}  y_{\alpha;\delta(\eps);\tau,k} \; \Phi_{\alpha}^{(\eps)} (  \xi)\;,
\end{equ}
and 
\begin{equ}
\Phi_{\alpha}^{(\eps)}(\xi) =   \prod_{i=1}^\infty \eps^{\frac{\alpha_i}{2}} H_{\alpha_i}   \big(\eps^{-\frac12}  \langle \xi, e_i \rangle \big) \;.  
\end{equ}
Now it is easy to check, using \eqref{e:yconverge} that for $k < K_\tau$ the functions $\Phi^{(\eps)}_{N;\tau,k} $ converge to $0$ and the functions $\Phi^{(\eps)}_{N;\tau,K_\tau} $ converge to $({\FFF}_{N;\tau,K_\tau})_{\hom}$, in both cases uniformly on bounded subsets of $B$. This implies in particular that the slightly weaker condition (2) in Lemma~\ref{le:contraction1} holds.

Furthermore, it is well known (see e.g. \cite[Chapter 4]{LedouxIso}) that under $\mu$, the random vectors $\sqrt{\eps}  \xi$ satisfy a large deviation principle on $B$ with rate $\eps$  and rate function 
\begin{equ}
\II_{\xi}(h) = 
\begin{cases}
\frac12 | h|_{\CH}^2 \qquad &\text{for $h \in \CH$,}\\
+\infty \qquad &\text{else.} 
\end{cases}
\end{equ}
Hence, the Lemma follows from the contraction principle, Lemma \ref{le:contraction1}. 
\end{proof}

\begin{lemma}\label{le:compact} 
Let $\II$ be given by \eqref{e:rateWiener} and let $\II_N$ be given by \eqref{e:defIIn}. Then $\II \colon \EE \to [0,\infty]$ is lower semicontinuous with compact sublevel sets. Furthermore for any closed set $\CC \subseteq \EE$ we have
\begin{equ}[e:RateFunctionLimit]
\lim_{\lambda \to 0}  \liminf_{N \to \infty} \inf_{\sz \in \CC_\lambda}  \II_N(\sz)   = \inf_{\sz  \in \CC}  \II(\sz)  \;.
\end{equ} 
Here $\CC_\lambda \eqdef \{ \sz \in \EE \colon  \dist_{\EE} (\sz, \CC) \leq \lambda \}$.
\end{lemma}
\begin{proof}[Proof of Lemma \ref{le:compact}]
As a consequence of \eqref{e:Fhomconvergence}, the assumptions of
\cite[Lemma 2.1.4]{DS} are satisfied, from which the first statement immediately follows.
The estimate \eqref{e:RateFunctionLimit} then follows from
the first part of the proof of \cite[Lemma 2.1.4]{DS}.
\end{proof}
\begin{lemma}[Exponential equivalence]\label{le:expontialequivalent}
We have for every $\lambda>0$
\begin{equ}
 \limsup_{N \to \infty }\limsup_{\eps \to 0 }\eps \log \mu\big( \xi \colon \| \FF^{(\eps)}_{\delta(\eps)}(\xi)  - \FF_{N;\delta(\eps)}^{(\eps)}(\xi) \|_{\EE} \geq \lambda \big) = - \infty \;.
\end{equ}

\end{lemma}
\begin{proof}
We write
\begin{equs}
 \int \|{\FFF}_{\delta(\eps); \tau}(\xi)  - {\FFF}_{N;\delta(\eps);\tau}(\xi) \|_{E_\tau}^2 &\mu (d \xi)  
 \leq 3 \int \| {\FFF}_{\delta(\eps); \tau}(\xi)  - {\FFF}_{\tau}(\xi) \|_{E_\tau}^2 \mu(d \xi)  \\
&  + 3 \int  \| {\FFF}_\tau(\xi)  - {\FFF}_{N;\tau}(\xi) \|_{E_\tau}^2 \mu(d \xi)  \\
&  + 3\int \| {\FFF}_{N;\tau}(\xi) - {\FFF}_{N;\delta(\eps),\tau}(\xi) \|_{E_\tau}^2  \mu( d \xi)\;.
\end{equs}
The assumption \eqref{e:convergenceWiener} on convergence of the $\FF_\delta$ to $\FF$ implies that the first term on the right hand side goes to zero as $\eps \to 0$. The conditional expectation is a contraction on $L^2(\mu,E_\tau)$, and hence the third term goes to zero as well. In particular, for fixed $N$ we can find an $\eps_N$ such that for $\eps < \eps_N$ the right hand side can be bounded by $4 \int  \| {\FFF}_\tau(\xi)  - {\FFF}_{N;\tau}(\xi) \|_{E_\tau}^2 \mu(d \xi)$. As a consequence of the integrability properties of random variables
belonging to a fixed Wiener chaos \cite[Theorem 4.1]{Borell3}, we conclude from this 
that there exists a sequence  $\beta_N$ increasing to $\infty$ as $N \to \infty$  and a constant $0< C < \infty$ 
such that for every $N$ and for every $\eps < \eps_N$ we have
\begin{equ}
\int \exp\Big( \beta_N \big\| {\FFF}_{\delta(\eps); \tau}(\xi)  - {\FFF}_{N;\delta(\eps);\tau}(\xi)\big\|_{E_\tau}^{\frac{2}{K_\tau}}\Big) \mu(d \xi)   \; \leq C \;.
\end{equ} 
Using Chebyshev's inequality we get for every $N$ and every $0<\eps < \eps_N$ that
\begin{equs}
\mu\Big(\xi& \colon \eps^\frac{{K_\tau}}{2} \big\|    {\FFF}_{\delta(\eps),\tau}(\xi)  - {\FFF}_{N;\delta(\eps);\tau} (\xi)   \big\|_{E_\tau} \geq \lambda  \Big) \\
&= \mu \bigg(\xi \colon  \exp\Big(   \beta_N \big\|    {\FFF}_{\delta(\eps),\tau}(\xi)  - {\FFF}_{N;\delta(\eps);\tau} (\xi)   \big\|_{E_\tau}^{\frac{2}{K_\tau}}  \Big) \geq  \exp\Big(\beta_N  \lambda^{\frac{2}{K_\tau}} \eps^{-1} \Big) \bigg) \\
&\leq   C \exp\Big(- \beta_N  \lambda^{\frac{2}{K_\tau}} \eps^{-1} \Big)\;.
\end{equs}
So the result follows.
\end{proof}

With Lemma ~\ref{le:LDPFN}, Lemma~\ref{le:compact}, and Lemma~\ref{le:expontialequivalent} in hand Theorem~\ref{thm:wienerLDP} follows easily. 
\begin{proof}[Proof of Theorem~\ref{thm:wienerLDP}]
It was already shown in Lemma~\ref{le:compact} that the rate function $\II$ satisfies the first condition in Definition~\ref{def:LDP}, so it remains to establish the upper bound \eqref{e:LDPupper0} and the lower bound \eqref{e:LDlower0}. 

Let $\CC \subset \EE$ be closed. Then we have for every $\lambda>0$ and $N \in \N$ that 
\begin{equs}
 \mu\big( \xi \colon  \FF^{(\eps)}_{\delta(\eps)}(\xi) \in \CC \big) &\leq \mu\big( \xi \colon  \FF^{(\eps)}_{N;\delta(\eps)}(\xi) \in \CC_\lambda \big) \\
&\qquad + \mu\big( \xi \colon \| \FF^{(\eps)}_{\delta(\eps)}(\xi)  - \FF_{N;\delta(\eps)}^{(\eps)}(\xi) \|_{\EE} \geq \lambda \big)\;,
\end{equs}
where as above $\CC_\lambda\eqdef \{ \mathbf{x} \colon \dist_{\EE}(\mathbf{x},\CC) \leq \lambda   \}$. According to Lemma~\ref{le:expontialequivalent} for any fixed $\lambda>0$ we can choose $N$ large enough (depending on $\CC$) to ensure
\begin{equ}
\limsup_{\eps \to 0}\eps \log  \mu\big( \xi \colon \| \FF^{(\eps)}_{\delta(\eps)}(\xi)  - \FF_{N;\delta(\eps)}^{(\eps)}(\xi) \|_{\EE} \geq \lambda \big)  \leq - \inf_{\sz \in \CC} \II(\sz)\;.
\end{equ}
On the other hand $\CC_\lambda$ is closed, so that Lemma~\ref{le:LDPFN} implies that for any fixed $N$
 \begin{equ}
\limsup_{\eps \to 0}\eps \log   \mu\big( \xi \colon  \FF^{(\eps)}_{N;\delta(\eps)}(\xi) \in \CC_\lambda \big) \leq - \inf_{\sz \in \CC_\lambda} \II_N(\sz)\;. 
 \end{equ}
Letting first $N$ go to $\infty$ and then $\lambda $ to $0$ and applying Lemma~\ref{le:compact} finishes the proof of the upper bound \eqref{e:LDPupper0}.

In a similar way, let $\CO \subset \EE$  be open and $\sz \in \CO$ and $\lambda >0$ be such that  $\{ \mathbf{x} \in \EE \colon \| \mathbf{x}  - \sz\|_{\EE} \leq 2 \lambda \} \subseteq \CO$. Then we have 
\begin{equs}
 \mu\big( \xi \colon  \FF^{(\eps)}_{\delta(\eps)}(\xi) \in \CO \big) &\geq \mu\big( \xi \colon  \| \FF^{(\eps)}_{N;\delta(\eps)}(\xi)  - \sz \|_{\EE} < \lambda \big) \\
&\qquad - \mu\big( \xi \colon \| \FF^{(\eps)}_{\delta(\eps)}(\xi)  - \FF_{N;\delta(\eps)}^{(\eps)}(\xi) \|_{\EE} \geq \lambda \big)\;.
\end{equs}
As above, Lemma~\ref{le:expontialequivalent} implies that for fixed $\lambda>0$ and $N$ large enough (depending on $\sz$) we have
\begin{equ}
\limsup_{\eps \to 0}\eps \log  \mu\big( \xi \colon \| \FF^{(\eps)}_{\delta(\eps)}(\xi)  - \FF_{N;\delta(\eps)}^{(\eps)}(\xi) \|_{\EE} \geq \lambda \big)  \leq -  \II(\sz) -1\;,
\end{equ}
and we get from Lemma~\ref{le:LDPFN} that for any $N$
\begin{equ}
 \liminf_{\eps \to 0}\eps \log  \mu\big( \xi \colon  \| \FF^{(\eps)}_{N;\delta(\eps)}(\xi)  - \sz \|_{\EE} < \lambda \big)  \geq - \inf_{\{\| \mathbf{x} - \sz \| \leq \frac12 \lambda \}}  \II_N(\mathbf{x})\;.
\end{equ}
Letting again $N$ go to $\infty$ first and then $\lambda $ to zero, we obtain the required lower bound from Lemma~\ref{le:compact}. 
\end{proof}

\section{Large deviations for stochastic PDEs}
\label{s:AC-LDP}

In this section we apply our abstract large deviation result, Theorem~\ref{thm:wienerLDP}, to the specific setting of the models
arising in the solution theory for the stochastic Allen-Cahn equation. 
We start by recalling that  $\CW_-^{(3)} = \{\Xi, \<1>,\<2>,\<3>,\<22>,\<31>,\<32>\}$.
We set $\EE \eqdef \bigoplus_{\tau \in \CW_-^{(3)}} E_\tau$, where  for $\tau \in \CW_-^{(3)} \setminus \{ \Xi,\;\<1>\}$ the spaces $E_\tau$ are given by the closure of the 
set of smooth functions $(z,\bar z) \mapsto \bigl(\Pi_z \tau\bigr)(\bar z)$
in the topology given by the norms
\begin{equ}
\|\Pi \tau\| = \sup_{\phi, \lambda, z} \lambda^{-|\tau|} \bigl|\bigl(\Pi_z \tau\bigr)(\phi_{z}^\lambda)\bigr|\;, 
\end{equ}
where the supremum runs over $\lambda \in (0,1]$, $z \in [-1,T+1] \times \T^d$,
and $\phi$ runs over a suitable set of test functions as in \eref{e:distPi}. For $\Xi$ and $\<1>$ the topology is given by the norms listed in \eqref{e:boundsM02}.
Since these norms are weaker than the supremum norm over a compact set, it follows from the
Stone-Weierstrass theorem that smooth functions can be approximated by polynomials with rational coefficients, so that the spaces $E_\tau$ are separable.

Note that $\MM_-^{(d)}$ can be viewed as a (nonlinear and rather complicated) 
closed subset of $\EE$. In the case $d = 2$ one can simplify the situation considerably
by restricting oneself to $\tau \in \{\Xi, \<1>,\<2>,\<3>\}$ and by dropping the
$z$-dependence of the functions $\Pi \tau$. For the remainder of this section, we will only consider
the case $d = 3$, the case $d=2$ follows \textit{mutatis mutandis}.

Let $\CH$  be $L^2( [-2,T+2] \times \T^3 )$. The Gaussian measure $\mu$ with Cameron-Martin space $\CH$ can be realised on $B$, the closure of the space of smooth functions in the Besov space $\Cc^{-\frac{5}{2}- \kappa}$ for any $\kappa >0$, and the triple 
$(B,\CH,\mu)$ is an abstract Wiener space. 
Let $\rho$ be a mollifying kernel, i.e.\  $\rho \colon \R \times \R^3 \to \R$ is a smooth function with $\int \rho(z) \, dz =1$ and compact support contained in the unit ball of $\R \times \R^3$. Then for $\xi \in B$ and  for $\delta \in (0,1)$ set 
\begin{equ}[e:xidelta]
\xi_\delta(z) = \langle\xi, \rho_z^\delta \rangle\;,
\end{equ}
where $\rho_z^\delta$ is defined as above in \eqref{e:rescale}.
Note that under $\mu$ the random distribution $\xi$ is a realisation of space-time white noise and $\xi_\delta$ is a smooth approximation to $\xi$. 

\begin{remark}
One could also have considered different approximations, like for example regularisations in space only or in time only. These
could also be handled in the same way, but we restrict ourselves to space-time regularisations in order to be able to reuse
the convergence results obtained in \cite{Regularity}.
\end{remark}

Let $\Psi(\xi_\delta) = (\Pi^{\xi_\delta}, F^{\xi_\delta}) $ be the canonical model constructed from $\Pi^{\xi_\delta}_{y} \Xi = \xi_\delta$ as described above in \eref{e:admissible3} and \eqref{e:canonical} (we have added a superscript $\xi_\delta$ to make more explicit the dependence on the specific realisation of the noise). As discussed in Section~\ref{ss:models} and Section~\ref{ss:continuity} $(\Pi^{\xi_\delta}, F^{\xi_\delta})$ lives in $\MM^{(d)}$ and it is uniquely characterised by the \emph{minimal model} i.e. the $\Pi^{\xi_\delta}\tau$ for $\tau \in \CW_-^{(3)}$.

We show now that the mapping $\FF_\delta$ that maps the white noise $\xi$ to the minimal model fits exactly into the framework developed in Section \ref{s:WienerLDP}. For any $\delta >0$ set  $\FF_{\delta} = \bigoplus_{\tau \in \CW_-^{(3)}}\FFF_{\delta; \tau}$  where  $\FFF_{\delta;\tau}(\xi) \eqdef \Pi^{\xi_\delta} \tau $ as above. By construction, for a fixed realisation of the noise $\xi$, the distributions $\phi \mapsto \Pi^{\xi_\delta}_{z} \tau (\phi)$ can be identified with a smooth function in the two variables $z$ and $\bar{z}$. In particular, it also takes values in $E_\tau$. 

It follows from the analysis in \cite[Sec. 10]{Regularity} that for each $\tau \in \CW^{(3)}_-$ and for $\delta >0$ the mapping $\xi \mapsto \FFF_{\delta;\tau}(\xi)$ belongs to the $E_\tau$-valued inhomogenous Wiener chaos of order $K_\tau$, where $K_\tau$ is the number of occurrences of the symbol $\Xi$ in $\tau$. More precisely, we have
\begin{equ}
K_{\Xi} = K_{\<1s>} =1\;, \quad K_{\<2s>}=2\;, \quad K_{\<3s>}=3 \;, \quad  K_{\<22s>} =   K_{\<31s>} = 4 \;,\quad K_{\<32s>} = 5\;.
\end{equ}
 
 We briefly outline this analysis.  Let $\CH^{\otimes_{\mathbf{s}}k}$ denote the $k$-fold symmetric tensor power of $\CH$. We identify  $\CH^{\otimes_{\mathbf{s}}k}$ with the space of symmetric square integrable functions $\CW$ in $k$ arguments  $z_i \in [-2,T+2] \times \T^3$. Then it is well known (see e.g. \cite[Sec 1.1]{Nualartbook}) that for every $k \geq 1$ there is an isometry (up to a factor $\sqrt{k!}$)
 \begin{equ}
 I_k \colon \CH^{\otimes_{\mathbf{s}}k }\to \CH^{(k)}(\mu, \R)\;
 \end{equ}
given by iterated stochastic integrals: For all $\xi$ in a set of $\mu$ measure one we have 
\begin{equ}
 I_k(\CW)(\xi)  = \int \!\! \cdots \!\! \int \CW (z_1, \ldots, z_k )  \;\xi(dz_1) \ldots \xi(dz_k) \;.
\end{equ}
We can thus identify every random variable in $\CH^{(k)}(\mu,\R)$ with a symmetric kernel in $k$ arguments. We can also apply $I_k$ to arbitrary kernels in  $L^2 ([-2, T+2] \times \T^3)^{\otimes k}$ by precomposing with the symmetrisation map
 \begin{equ}[e:symmetrisation]
 \CW^{(k)}(z_1, z_2, \ldots, z_k) \mapsto \frac{1}{k!}\sum_{\sigma } \CW^{(k)}(z_{\sigma(1)}, \ldots, z_{\sigma(k)})\;,
 \end{equ} 
 where the sum runs over all elements $\sigma$ in the group of permutations of $\{1, \ldots, k\}$. The symmetrisation \eqref{e:symmetrisation} is a contraction on $L^2 ([-2, T+2] \times \T^3)^{\otimes k}$. So we obtain
 \begin{equ}[e:L2iso]
 \int  \big( I_k(\CW^{(k)}) (\xi ) \big)^2 \, \mu(d\xi)  \leq k ! \int \CW^{(k)}(z_1, z_2, \ldots ,z_k)^2 \,dz_1 \ldots \, dz_k \,.
 \end{equ}
 
It is shown recursively in \cite[Sec. 10]{Regularity}, that for any $\tau \in \CW^{(3)}_-$ and for $\delta >0$,  
$\Pi^{\xi_\delta}_{y}(\tau)$  can be characterised by kernels $\CW_{\delta;\tau,k }(\bar{y}, y, \cdot) \in \CH^{\otimes k}$ for $k = K_{\tau}, K_{\tau}-2, K_{\tau}-4, \ldots$. The precise form of the kernels can be read off directly from the graphs  $\<1>, \<2>, \<3>$ etc. in \eref{e:trees1} above (in fact, this is the reason for this notation).

We illustrate this construction for $\tau =\<32>$, which is the most involved term. All other
symbols can be dealt with in a very similar (but easier) way. The kernel $ \CW_{\delta, \<32s> ,5} (\bar{y},y, \cdot)$ for the contribution in the highest Wiener chaos $K_{\<32s>} = 5$ is given by
  \begin{equs}\label{e:Kernel32}
 \CW_{\delta, \<32s> ,5} (\bar{y},y; z_1, \ldots, z_5)\eqdef& \int dx  \, K_\delta(z_1 - x) \, K_\delta(z_2 -x) \, K_\delta (z_3 -x)\\
 & \big( K(x - \bar{y}) - K(x- y) \big) \, K_\delta(z_4 -y) \, K_\delta(z_5 -y) \;.
 \end{equs}
Here  $K$ denotes the kernel introduced in Section~\ref{ss:models} and $K_\delta \eqdef K \star \rho_\delta$. The connection between \eqref{e:Kernel32} and the graphical representation $\<32>$ is the following: Each vertex of $\<32>$ corresponds to a variable in $[-2,T+2] \times \T^3$- the leafs are $z_1, \ldots , z_5$, the root corresponds to $\bar{y}$ and $y$ and the central vertex corresponds to $x$. Then as explained in  \eqref{e:trees1} every down facing line in $\<32>$ corresponds to one occurrence of a kernel $K$ or $K_\delta$ in \eqref{e:Kernel32}: If the line is connected to a leaf, the kernel is $K_\delta$ and the kernel is $K$ otherwise. The variable $x$ corresponding to a vertex that are neither a leaf nor the root is integrated out.  The fact that we have to subtract one occurrence of the kernel in $\big( K(x - \bar{y}) - K(x- y) \big)$  corresponds to the second term in \eqref{e:admissible3} which has to be added to guarantee for the right behaviour near the diagonal $\bar{y}=y$ (see \cite[Sec. 6 and Sec. 10]{Regularity} for details).

 If $h \in \CH$ is an $L^2$-function then it is easy to check from the recursive construction in  \eref{e:admissible3} and \eqref{e:canonical} that for every test function $\phi$ the iterated integral 
 \begin{equ}[e:Whomogen]
 \int   \langle \CW_{\delta, \<32s> ,5} (\cdot,y, z_1, \ldots z_5) , \phi \rangle  \, h(z_1) \ldots h(z_5) \, dz_1 \ldots \, dz_5
 \end{equ}
 yields the value $\Pi^{h_\delta}_{y}   \<32> (\phi)$ for the canonical model constructed from $\Pi^{h_\delta}_{y} \Xi = h_\delta \eqdef h \star \rho_\delta$. Here we have set
\begin{equ}[e:explanationbracket]
\langle\CW_{\delta, \<32s> ,5} (\cdot,y, z_1, \ldots z_5), \phi  \big\rangle \eqdef \int \CW_{\delta, \<32s> ,5} (\bar{y},y, z_1, \ldots z_5)\, \phi(\bar{y})   \,d\bar{y}\;.
\end{equ}

 In the case of the regularised white noise $\xi_\delta$ the relationship between the kernels and the canonical model is slightly more complicated. The iterated stochastic integral
 \begin{equ}[e:ItIt4]
 I_5(\langle \CW_{\delta, \<32s> ,5} (\cdot,y, \ldots)  , \phi \rangle   )   = \int \langle   \CW_{\delta, \<32s> ,5}  (\cdot,y, z_1, \ldots z_5), \phi \rangle  \, \xi(d z_1) \ldots \xi(d z_5) 
 \end{equ}
does not give the value $\Pi^{\xi_\delta}_{y}   \<32> (\phi)$ for the canonical model contracted from $\xi_\delta$. Indeed, iterated integrals obey an \emph{It\^o-type} chain rule that produces some extra terms. These terms can be represented by additional kernels $\CW_{\delta, \<32s> ,3} (\bar{y},y, z_1, z_2, z_3)$ and  $\CW_{\delta, \<32s> ,1} (\bar{y},y, z_1)$. These kernels are obtained through suitable \emph{contractions} of the graph $\<32>$. Their precise form is not relevant for our discussion and we refer to \cite[Sec.~10.5]{Regularity} and \cite{Nualartbook} for more detailed explanations.  Once these kernels are constructed, the $\Pi_{y}^{\xi_\delta} \<32>$ is given as
\begin{equ}[e:itint]
 \Pi_{y}^{\xi_\delta}\<32>(\phi)\eqdef  \sum_{k \in\{ 5,3, 1   \}}  \,I_k\big(    \langle \CW_{\delta;\<32s>,k }(\cdot, y, \cdots), \phi \rangle \big) \;. 
\end{equ}
As above in \eqref{e:explanationbracket} the pointed bracket on the right hand side denotes integration of the kernels $\CW_{\delta;\tau,k }$ against $\phi$ in the first variable $\bar{y}$, whereas $I_k$ represents the iterated stochastic integration with respect to the remaining variables $z_i$.

The decomposition \eref{e:itint} is useful for several reasons. On one hand, it gives a natural decomposition of $\Pi^{\xi_\delta}\<32>$ into its components in homogenous Wiener chaoses. 
Furthermore, \eqref{e:itint}  in conjunction with the estimate \eqref{e:L2iso} reduces the problem of estimating the second moment  of the expression $ \Pi_{y}^{\xi_\delta}\<32> (\phi) $ to deriving bounds on the $L^2$-norm of the kernels $\CW_{\delta;\<32s>,k }$. Using the equivalence of all moments in a fixed Wiener chaos, these bounds in turn can then be used as input into a \emph{Kolmogorov-like theorem} (see \cite[Thm 10.7]{Regularity}) to obtain bounds on  $\int\|\Pi^{\xi_\delta} \tau(\xi) \|_{E_{\tau}}^p \mu(d\xi)$ for arbitrary $1<p<\infty$.
 
On the other hand, we can immediately read off from \eqref{e:Whomogen} and \eqref{e:itint} the homogeneous part of $\FFF_{\delta;\<32s>}$. According to the discussion in Section~\ref{s:WienerLDP}, the terms corresponding to the kernels $\CW_{\<32s>, k}$ for $k = 1,3$ do not influence the value of the $(\FFF_{\delta;\<32s>})_{\hom}$. For $k = K_{\<32s>} =5$ the construction of the iterated integrals (see \cite[Section 1.1]{Nualartbook}) yields an explicit expression of the quantity $\FFF_{\delta; \<32s>,5}$ evaluated at a shifted noise $\xi+h$ for $h \in \CH$. For any $h \in \CH$ and every test function $\phi$ we get for $\mu$-almost every $\xi$ 
\begin{equs}\label{e:iintplush}
\FFF_{\delta; \<32s>,5}(\xi+h)[y, \phi]   = \int  \big\langle& \CW_{\delta, \<32s> ,5} (\cdot, y, z_1, \ldots z_5), \phi  \big\rangle \,\\ 
&\big(\xi(d z_1) + h(z_1) dz_1 \big)  \cdots \big(\xi(d z_5) + h(z_5) \, dz_5 \big)\;.
\end{equs}
Here we use square brackets to denote the evaluation of $\FFF_{\delta; \<32s>,5}(\xi+h) \in E_{\<32s>}$ at $y$ and at a test function $\phi$.
According to  \eqref{e:fhom} $(\FFF_{\delta;\<32s>})_{\hom}(h)$ evaluated at $(y, \phi)$ is given by the expectation of this expression.  But all of the iterated integrals that involve at least one power of $\xi$ have vanishing expectation. The only term with non-zero expectation with respect to $\mu$ is given by 
\begin{equ}
 \int   \langle \CW_{\delta, \<32s> ,5} (\cdot,y, z_1, \ldots z_5) , \phi \rangle  \, h(z_1) \ldots h(z_5) \, dz_1 \ldots \, dz_5\;.
\end{equ}
According to the discussion around \eqref{e:Whomogen} this is precisely $\Pi_y^{h_\delta}\<32>(\phi)$ for the canonical model  constructed from $h_\delta$. 
 
The reasoning for all the other $\tau \in \CW$ is similar -- we refer to \cite[Sec. 10]{Regularity} for the precise expressions of the kernels $\CW_{\delta; \tau, k}$. In each case the $\FFF_{\delta;\tau,k}$ live in $\CH^{(k)}(\mu, E_{\tau})$ and in each case the homogenous part  
$(\FFF_{\delta;\tau})_{\hom}(h)$ is given by the canonical model constructed from $h_\delta$.  
 
 Finally,  the recursive definition  \eref{e:admissible3} and \eqref{e:canonical} is homogenous in the input $\xi$. Replacing the building block $\xi$ by $\sqrt{\eps} \xi$ produces $\Pi^{\sqrt{\eps} \xi_\delta} \tau =  \eps^{\frac{K_{\tau}}{2}} \Pi^{ \xi_\delta} \tau$ almost surely. This is precisely the rescaling assumed in the abstract setting in  \eqref{e:scaling}. Hence, we are in the situation described in Section~\ref{s:WienerLDP} and we can apply Theorem~\ref{thm:wienerLDP} directly to the functions $\FF^{(\eps)}_\delta$ for fixed $\delta$ to obtain the following large deviation result. 
 
\begin{theorem}\label{thm:trivialLDP}
For any $\eps, \delta>0$ let $ (\Pi^{\sqrt{\eps}\xi_\delta}, F^{\sqrt{\eps} \xi_\delta} ) \in \MM^{(3)}$ be the canonical model defined recursively via \eref{e:admissible3} and \eqref{e:canonical} from 
$\Pi^{\sqrt{\eps}\xi_\delta}_{y} \Xi = \sqrt{\eps}\xi_\delta$. In the same way, for $h \in \CH$ let $ (\Pi^{h_\delta},F^{h_\delta})$ be the canonical model constructed from $h_\delta = h \star \rho_\delta$.  

Then for any fixed value of $\delta >0$  the distributions of $\{(\Pi^{\sqrt{\eps}\xi_\delta}, F^{\sqrt{\eps} \xi_\delta} ), \eps > 0 \}$ under $\mu$ satisfy a large deviation principle on $\MM^{(3)}$ with rate $\eps$ and rate function
\begin{equ}
\II_{\mathrm{Model\, }\delta}\big( (\Pi ,F) \big) \eqdef \inf \Big\{ \tfrac12 | h |_{\CH}^2 \colon \, (\Pi^{h_\delta},F^{h_\delta}) = (\Pi,F) \Big\} \;.
\end{equ}
\end{theorem}
\begin{proof}
Theorem~\ref{thm:wienerLDP} applied to $\{ \FF^{(\eps)}_\delta; \eps >0 \}$ for fixed value of $\delta$ yields a large deviation principle for the  random variable $\bigoplus_{\tau \in \CW} \Pi^{\sqrt{\eps} \xi_\delta} \tau \in \EE$. For $\sz = \bigoplus_{\tau \in \CW} s_\tau$ the rate function is given by 
\begin{equ}
\widetilde{\II_{ \delta}}(\sz) = \inf \Big\{ \tfrac12 | h |_{\CH}^2 \colon \Pi^{h_\delta}\tau = s_\tau \, \text{for } \tau \in \CW^{(3)}_-  \Big\}\;.
\end{equ}
By construction this random variable takes values in $\MM_-^{(3)}$ almost surely. As $\MM_-^{(3)}$ is a closed subset of $\EE$ it follows automatically that the large deviation principle also holds on $\MM_-^{(3)}$ with respect to the relative topology.

According to Theorem~\ref{theo:extension} the unique extension to an admissible model is continuous from $\MM_-^{(d)}$ to $\MM^{(d)}$. Hence  the contraction principle, Lemma~\ref{le:contraction1},  implies a large deviation principle for the full model. The rate function is given by 
\begin{equs}
\II_{\mathrm{Model\, }\delta} \big( (\Pi ,F) \big)  &= \inf \{ \widetilde{\II_\delta}  (\sz) \colon \, \sz \in \CM^{(3)}_- \, \text{ extension of $\sz$ is $(\Pi,F)$}     \} \\
&= \inf  \Big\{ \tfrac12 | h |_{\CH}^2 \colon  \, \Pi^{h_\delta}\tau = s_\tau \, \text{for } \tau \in \CW^{(3)}_-   \, \text{ and}\\
& \qquad \qquad \qquad \text{ extension of $\sz = \bigoplus s_\tau$ is $(\Pi, F)$} \Big\}\\
&= \inf \Big\{ \tfrac12 | h |_{\CH}^2 \colon \, (\Pi^{h_\delta},F^{h_\delta}) = (\Pi,F) \Big\} \;,
\end{equs}
which is precisely our claim.
\end{proof}

 Combining this result with the contraction principle, Lemma~\ref{le:contraction1} and Theorem~\ref{theo:gen} we recover the  Freidlin-Wentzell type bounds for the stochastic PDE \eqref{e:AC} driven by a noise term $\sqrt{\eps} \xi_\delta$ for fixed correlation length $\delta$ as the noise strength $\eps$ goes to zero. Of course, such bounds can also be derived without any reference to the theory of regularity structures.

 Our approach produces meaningful new results only when considering the limit $\delta \to 0$.  Indeed, if we could show that the $\FF_\delta$ converges to a limit $\FF$ in the sense of \eqref{e:convergenceWiener} we could apply Theorem~\ref{thm:wienerLDP} to obtain a large deviation principle for any diagonal sequence $\delta, \eps \to 0$. Unfortunately, as already explained above in Section~\eqref{sec:renorm} this is not true.  However,  if we replace the models $\Pi^{\xi_\delta}$ by the \emph{renormalised} models  $\hat{\Pi}^{\xi_\delta}$ as defined above in \eqref{e:3dWick} we  obtain uniform bounds and can pass to the limit $\delta \to 0$. 

More precisely, for any $\eps >0$ and  $\delta>0$, let $(\hat{\Pi}^{\sqrt{\eps}\xi_\delta}, \hat{F}^{\sqrt{\eps} \xi_\delta} ) \in \MM^{(3)}$ be the model defined recursively via \eref{e:admissible3}, \eqref{e:Wick}, and 
\eqref{e:3dWick} from $\hat{\Pi}^{(\eps)}_{y} \Xi = \sqrt{\eps}\xi_\delta$, and then 
extended to all of $\MM^{(3)}$ by Theorem~\ref{theo:extension}.
For $\delta=0$, we write $(\hat{\Pi}^{\sqrt{\eps}\xi}, \hat{F}^{\sqrt{\eps} \xi} ) = (\hat{\Pi}^{\sqrt{\eps}\xi_0}, \hat{F}^{\sqrt{\eps} \xi_0} ) \in \MM^{(3)}$ for the limit as $\delta \to 0$ of these models,
which exists by Theorem~\ref{theo:mainregularity}.
As above, for $h \in \CH$, we write $(\Pi^{h},F^{h})$ for the canonical model constructed from $h$.  
We then have the following result.

 \begin{theorem}\label{thm:main0}
Let $\eps \mapsto \delta(\eps) \geq 0$ be any function satisfying $\lim_{\eps \to 0}   \delta(\eps) =0$. Then the distributions of  $\{(\hat{\Pi}^{\sqrt{\eps}\xi_{\delta(\eps)}}, \hat{F}^{\sqrt{\eps} \xi_{\delta(\eps)}} ), \eps > 0 \}$ under $\mu$ satisfy a large deviation principle on $\MM^{(3)}$ with rate $\eps$ and rate function
\begin{equ}
\II_{\mathrm{Model}}  \big( (\Pi ,F) \big) \eqdef \inf \Big\{ \tfrac12 | h |_{\CH}^2 \colon \, (\Pi^{h},F^{h}) = (\Pi,F) \Big\} \;.
\end{equ}
 \end{theorem}
\begin{proof}
For $\delta>0$, define the $\EE$-valued functions $\hat{\FF}_{\delta}(\xi)$ as  $\hat{\FFF}_{\delta;\tau}(\xi) = \hat{\Pi}^{\xi_\delta}\tau$. It is shown in \cite[Thms~10.7 and 10.22]{Regularity} that for every $\tau \in \CW^{(3)}_-$ these $\hat{\FFF}_{\delta;\tau}(\xi)$ converge to the limit $\hat{\FFF}_{\tau}(\xi)$ in the sense that
\begin{equ}
\lim_{\delta \to 0 }\int   \big\| \hat{\FFF}_{\delta;\tau}(\xi) - \hat{\FFF}_{\tau}(\xi) \big\|_{E_\tau}^2  \mu(d \xi)   = 0\;.
\end{equ}
Furthermore, each $\hat{\FFF}_\tau(\xi)$ is given by iterated stochastic integrals against kernels $\hat{\CW}_{\tau; k}$  as, for example, above in \eqref{e:ItIt4}. For each $\tau \in \CW^{(3)}_-$ the contribution in the highest Wiener chaos  $\CH^{(K_\tau)}(\mu, E_\tau)$ is obtained by simply removing the convolution with $\rho_\delta$ in the definitions of the the integral kernels $K_\delta$. For example, we have
  \begin{equs}
 \hat{\CW}_{ \<32s> ,5} (\bar{y},y, z_1, \ldots, z_5)&\eqdef\int dx  \, K(z_1 - x) \, K(z_2 -x) \, K (z_3 -x)\label{e:Kernel32prime}\\
 & \big( K(x - \bar{y}) - K(x- y) \big) \, K(z_4 -y) \, K(z_5 -y) \;.
 \end{equs}
 This is because the renormalisation procedure described in \eqref{e:Wick} and \eqref{e:3dWick} only affects the contributions in $\CH^{(k)}(\mu, E_\tau)$ for $k < K_\tau$. This implies that for each $\tau \in \CW^{(3)}_-$ the homogeneous part $\big( \hat{\FF}_{\tau} \big)_{\hom}(h)$ is given by $\Pi^{h} \tau$, the contribution at level $\tau$ for the canonical model constructed from $\Pi_z \Xi = h$. Hence, for any function $\delta(\eps) \geq 0$ with $\lim_{\eps \to 0} \delta(\eps) =0$, Theorem~\ref{thm:wienerLDP} yields a large deviation principle on $\EE$ for the sequence of random variables $\hat{\FF}_{\delta(\eps)}^{(\eps)}$.  From this we obtain a large deviation principle for the full renormalised 
 models $(\hat{\Pi}^{\sqrt{\eps}\xi_{\delta(\eps)}} , \hat{F}^{\sqrt{\eps}\xi_{\delta(\eps)}} )$ using Theorem~\ref{theo:extension} as well as the contraction principle, Lemma~\ref{le:contraction1}, in the same way as in the proof of Theorem~\ref{thm:trivialLDP}
\end{proof}

Applying Theorem~\ref{theo:gen} and the contraction principle, Lemma~\ref{le:contraction1}, once more, we immediately obtain our main result. In order to formulate it, we use the notation 
\begin{equ}
\CC_{u_0}([0,T],\CC^\eta) \eqdef \{ u \in \CC ([0,T],\CC^\eta) \colon u(0, \cdot) = u_0 \}\;,
\end{equ}
and $\mathcal{X}_{u_0} \eqdef \CC_{u_0}([0,T],\CC^\eta) \cup \{\infty\}$. As above we endow $\mathcal{X}_{u_0}$ with a metric by postulating that $\{\infty \}$ is at distance $1$ from every point in $\CC_{u_0}([0,T],\CC^\eta)$.

\begin{theorem}\label{thm:main1}
Let $\eta \in (-3/2, 1-{d\over 2} - \kappa)$ as before
and, for any $\eps, \delta >0$ and for any initial datum $u_0 \in \CC^{\eta}$, let $\hat{u}^{(\eps)}_\delta$ be the unique solution of the SPDE
\begin{equ}[e:ACend]
\d_t u = \Delta u + (C + 3 \eps C_\delta^{(1)} - 9 \eps^2 C_\delta^{(2)}) u - u^3 + \sqrt{\eps}\xi_\delta\;
\end{equ}
with initial datum $u_0$. 
Here $C \in \R$ is arbitrary and the constants $C^{(i)}_\delta$ are defined in \eqref{e:CWick} and \eqref{e:CWick7}. For $\eps>0$, let $\hat{u}^{(\eps)}_0$ be the limit as $\delta \to 0$ of these solutions, constructed in theorems~\ref{theo:gen} and \ref{theo:mainregularity}.
(As before we set it equal to $\infty$ if the maximal existence time is less or equal to $T$.)

Let $\eps \mapsto \delta(\eps) \geq 0$ be a function with 
\begin{equ}
\lim_{\eps \to 0 } \delta(\eps)=0\;.
\end{equ}
Then the sequence $\hat{u}^{(\eps)}_{\delta(\eps)}$ satisfies a large deviation principle in $\mathcal{X}_{u_0}$ with rate function $\II(\infty) = +\infty$ and
\begin{equ}[e:LDPfinal]
\II(u) =  \frac12 \int_0^T \!\!\int_{\T^d} \big( \partial_t u -  \Delta u - C u + u^3 \big)^2 \, dx \, dt   \;.
\end{equ}
This is with the understanding that $\II(u) = +\infty$ if $u \notin L^3([0,T] \times \T^3)$ 
or if the distribution appearing 
in the right hand side of \eqref{e:LDPfinal} is not a square integrable function.
\end{theorem}

\begin{remark}
The fact that the prefactor for $C_\delta^{(1,2)}$ in \eqref{e:ACend} is given
by $\eps $ and $\eps^2$ respectively follows from \eqref{e:scaling} since
$K_{\<2s>} = 2$ and $K_{\<22s>} = 4$. This is also consistent with
the definitions (10.35) and (10.41) in \cite{Regularity}.
\end{remark}

\begin{proof}
According to Theorem \ref{theo:mainregularity} the solution of \eqref{e:ACend} can be obtained by first applying the solution map $\CS_A$  to the renormalised model $(\hat{\Pi}^{\sqrt{\eps} \xi_{\delta(\eps)}}, \hat{F}^{\sqrt{\eps} \xi_{\delta(\eps)}})$.

According to Theorem~\ref{thm:main0} the renormalised models 
$(\hat{\Pi}^{\sqrt{\eps} \xi_{\delta(\eps)}}, \hat{F}^{\sqrt{\eps} \xi_{\delta(\eps)}})$ satisfy a 
large deviation principle on $\MM^{(3)}$. Denote by $\MM^{(3)}_\CH \subset \MM^{(3)}$ the set of
models obtained by the canonical lift of some element $h \in \CH$. Since the solutions to
the Allen-Cahn equation driven by a noise term in $\CH$ are global, the map $ \CS_A$ does not
take the value $\infty$ on $\MM^{(3)}_\CH$.
Furthermore, as a consequence of 
Theorem~\ref{theo:gen}, the solution operator $\CS_A$ (for fixed $u_0$ and $C$) 
is continuous  from an open neighbourhood of $\MM^{(3)}_\CH$ in $\MM^{(3)}$ into $\CC_{u_0}([0,T],\CC^\eta)$.  
Since $\MM^{(3)}_\CH = \{\II_{\mathrm{Model}} < +\infty\}$, 
the contraction principle, Lemma~\ref{le:contraction1}, immediately implies that the 
$\hat{u}_{\delta(\eps)}^{(\eps)}$ satisfy a large deviation principle on  $\mathcal{X}_{u_0} =\CC_{u_0}([0,T],\CC^\eta)\cup \{\infty\}$. 

The rate function is given by 
\begin{equs}
\II(u) &= \inf\{  \II_{\mathrm{Model}}(\Pi, F) \colon  \CS_A (\Pi,F) =u \} \\
&= \inf \big\{  \tfrac12 | h |_{\CH}^2 \colon \,  \CS_A (\Pi^{h},F^{h}) = u  \big\}\;.
\end{equs}
According to Theorem~\ref{theo:mainregularity}, for $h \in \CH$ the function $ \CS_A (\Pi^{h},F^{h})$ is simply the classical solution to 
\begin{equ}
\d_t u = \Delta u + Cu - u^3 + h\;.
\end{equ}
Standard parabolic regularity theory implies that for $h \in L^2([0,T]\times \T^3)$ (and for such an irregular choice of initial condition) $u$ will at least attain values in $L^3([0,T] \times \T^3))$. Conversely, if $u$ can be obtained in this way from $h$
we can simply recover $h$ by setting
\begin{equ}
h = \d_t u - ( \Delta u + Cu - u^3)\;.
\end{equ}
This concludes the argument.
\end{proof}

\begin{remark} 
Actually, we do not expect the renormalised solutions to explode either for  $d=2$ or $d=3$. Our argument does not imply this, but we can conclude that the probability of finite time explosion decays faster than $e^{-c \eps^{-1}}$ for any $c \in (0,\infty)$.
\end{remark}

In a very similar way, Theorem~\ref{thm:main0} implies the following large deviation principle for solutions of the SPDE without renormalisation:
\begin{theorem}\label{thm:main2}
Let $\eta \in (-3/2, 1-{d\over 2} - \kappa)$ as before. For any $\eps, \delta >0$ and for any initial datum $u^0 \in \CC^{\eta}$ let $u^{(\eps)}_\delta$ be the unique (global) solution of the SPDE
\begin{equ}[e:ACend2]
\d_t u = \Delta u + C u - u^3 + \sqrt{\eps} \xi_\delta\;,
\end{equ}
with initial datum $u_{0}$. 

Let $\eps \mapsto \delta(\eps) > 0$ be a function with $\lim_{\eps \to 0} \delta(\eps)=0$ and
\begin{equ}[e:lambdaassumption]
\lim_{\eps \to 0 } \eps \delta^{-1} = \lambda^2 \in [0,\infty).
\end{equ}
Then the $u^{(\eps)}_{\delta(\eps)}$ satisfy a large deviation principle in  $\CC([0,T],\CC^\eta)$ with rate function \begin{equ}
\II(u) =  \frac12 \int_0^T \!\!\int_{\T^d} \big( \partial_t u -  \Delta u + C_\lambda u + u^3 \big)^2 \, dx \, dt   \;.
\end{equ}
where $C_\lambda =  C- 3 \lambda^2  \int_{\R^4} \bigl(P\star \rho\bigr)(t,x)^2\,dt\,dx$. 
\end{theorem}
\begin{remark}
Recall that $P$ denotes the heat kernel on the $3$-dimensional torus without any truncation. (See also Remark~\ref{rem:rem214}.)
\end{remark}
\begin{proof}
We start by rewriting \eqref{e:ACend2} as
\begin{equ}
\d_t u = \Delta u + \widetilde{C}(\eps) u - \big(u^3 - \big( 3 \eps C_\delta^{(1)} - 9 \eps^2 C_\delta^{(2)}\big) u \big) + \sqrt{\eps} \xi_\delta
\end{equ}
where $\widetilde{C}(\eps) = C - 3 \eps C_\delta^{(1)} + 9 \eps^2 C_\delta^{(2)}$. According to Theorem~\ref{theo:mainregularity} the solution to this equation can be obtained by applying the solution operator $\CS_A$ corresponding to the choice $C =\widetilde{C}(\eps)$  to the renormalised model $(\hat{\Pi}^{\sqrt{\eps} \xi_\delta(\eps)}, \hat{F}^{\sqrt{\eps} \xi_\delta(\eps)})$. 

According to the assumption~\eqref{e:lambdaassumption} and to the definition \eqref{e:CWick} and \eqref{e:CWick7} of the constants $C_{\delta}^{(1)}$ and $C_\delta^{(2)}$ the  $\widetilde{C}(\eps)$ converge to $C_\lambda$. Now the argument proceeds like the proof of Theorem~\ref{thm:main1}, noting only that the solution map is locally uniformly continuous in the choice of $C$ when applying the contraction principle, Lemma~\ref{le:contraction1}.
\end{proof}

\begin{remark}
The analogous result for $d=2$ can be obtained in a similar way. There, according to \eqref{e:2DWicklog} we get  for the value of the modified constant $C_\lambda = C- \lambda^2   {3\over 4\pi}$ under the assumption that $ \lim_{\eps \to 0} \eps \log \delta(\eps)^{-1} = \lambda^2 \in [0,\infty)$. In particular, the value of $C_\lambda$ does not depend on the choice of mollifying kernel $\rho$.
\end{remark}

\begin{remark}
For $d=3$ we could also consider schemes of the form 
\begin{equ}
\d_t u = \Delta u + \widetilde{C}(\eps) u - \big(u^3 - 3 \eps C_\delta^{(1)}  u \big) + \sqrt{\eps} \xi_\delta,
\end{equ}
i.e. schemes that ignore the logarithmic sub divergence. In this case we recover a large deviation result with rate function depending on $\lim_{\eps \to 0}  \eps^2 \log \delta(\eps)^{-1} = \lambda^2 \in [0, \infty)$.
\end{remark}

\begin{remark}
Even though our statement does not make reference to regularity structures, the proof uses heavily the fact that there is a non-trivial renormalised limit as $\delta \to 0$  for fixed $\eps>0$. In particular, we do not know if a similar statement (but with stronger assumptions on how $\delta(\eps) \to 0$) holds true in four spatial dimensions where no renormalised solutions are available. 
\end{remark}

\bibliographystyle{./Martin}
\bibliography{./LDP}
\end{document}